\documentclass[11pt, a4paper]{amsart}
\usepackage{fullpage}
\usepackage{amsmath}
\usepackage{amssymb}
\usepackage{amsfonts}
\usepackage{mathrsfs}
\usepackage{amscd}
\usepackage{graphicx}
\usepackage[shortlabels]{enumitem}
\usepackage{mathtools}
\usepackage{tikz-cd}
\usepackage{hyperref}
\hypersetup{
    colorlinks=true,
    linkcolor = blue,
    citecolor= blue,
    urlcolor=cyan,
}
\usepackage[capitalise]{cleveref}
\usepackage{pstricks-add}
\usepackage{pgf,tikz}
\usepackage{subcaption}
\usepackage{bm}

\usepackage{tkz-tab}
\usepackage{xpatch}
\xpatchcmd{\tkzTabLine}{$0$}{$\bullet$}{}{}
\tikzset{t style/.style={style=solid}}

\usetikzlibrary{arrows}

\crefformat{section}{\S#2#1#3} 
\crefformat{subsection}{\S#2#1#3}
\crefformat{subsubsection}{\S#2#1#3}
\crefname{figure}{{\sc Figure}}{{\sc Figure}}
\crefname{equation}{}{}

\numberwithin{equation}{section}       

\theoremstyle{plain} 
\newcommand{\thistheoremname}{}
\newtheorem*{genericthm*}{\thistheoremname}
\newenvironment{namedtheorem*}[1]
  {\renewcommand{\thistheoremname}{#1}%
   \begin{genericthm*}}
  {\end{genericthm*}}

\newtheorem{theorem}{Theorem}[section]

\newtheorem{lemma}[theorem]{Lemma}
\newtheorem{corollary}[theorem]{Corollary}

\newtheorem*{claim*}{Claim}

\theoremstyle{definition}

\newtheorem{remark}[theorem]{Remark}

\theoremstyle{definition}

\newtheoremstyle{citing}
  {3pt}
  {3pt}
  {\itshape}
  {}
  {\bfseries}
  {.}
  {.5em}
  {\thmnote{#3}}
\theoremstyle{citing}
\newcommand{\N}{\mathbb{N}}

\newcommand{\R}{\mathbb{R}}

\newcommand{\cH}{\mathcal{H}}

\newcommand{\cM}{\mathcal{M}}
\newcommand{\cN}{\mathcal{N}}
\newcommand{\cO}{\mathcal{O}}
\newcommand{\cP}{\mathcal{P}}

\newcommand{\cT}{\mathcal{T}}

\newcommand{\teta}{\widetilde{\teta}}

\newcommand{\e}{\varepsilon}

\newcommand{\ov}{\overline}

\renewcommand{\=}{ : = }

\DeclareMathOperator{\HD}{HD} 

\tikzset{
  declare function={
    sgn(\x) = (and(\x<0, 1) * -1) +
    (and(\x>0, 1) * 1) +
    (and(\x==0, 1) * 0);
  }
}

\usepackage{lipsum}
\begin{document}


\usetikzlibrary{shapes, arrows, calc, arrows.meta, fit, positioning, quotes} 
\tikzset{  
    state/.style ={ellipse, draw, minimum width = 0.9 cm}, 
    point/.style = {circle, draw, inner sep=0.18cm, fill, node contents={}},  
    bidirected/.style={Latex-Latex,dashed}, 
    el/.style = {inner sep=2.5pt, align=right, sloped}  
}  

\colorlet{ColorGray}{gray!10}

\title{Towers and Bratteli-Vershik systems in Fibonacci-like unimodal maps}
\author{Jorge Olivares-Vinales} \thanks{} 
\address{Shanghai Center for Mathematical Sciences, Jiangwan Campus, Fudan University, No 2005 Songhu Road, Shanghai, China 200438 }
\email{jolivaresv@fudan.edu.cn}

\author{Semin Yoo}
\address{Discrete Mathematics Group \\ Institute for Basic Science \\ 55 Expo-ro Yuseong-gu, Daejeon 34126 \\ South Korea}
\email{syoo19@ibs.re.kr}

\subjclass[2020]{37E05, 37E15, 37E25}
\keywords{Fibonacci-like unimodal maps, minimal Cantor systems}

\begin{abstract}
  For a class of Fibonacci-like unimodal maps, the restriction to the $\omega$-limit set of the unique turning point defines a minimal Cantor system. We construct these Cantor sets geometrically using a nested sequence of finite covers with a tower structure. From this tower structure, we recover the associated Bratteli–Vershik model determined by the cutting times and obtain an explicit formula for the unique ergodic invariant probability measure supported on the $\omega$-limit set. We conclude with applications illustrating the scope of the construction.
  
\end{abstract}

\maketitle
\section{Introduction}

Given a unimodal map $f \colon [-1,1] \longrightarrow [-1,1]$, the $\omega$-limit set of its turning point is a 
compact set that is forward invariant under $f$. There are many combinatorial conditions that imply that the $\omega$-limit set of the turning 
point is a Cantor set on which $f$ acts in an interesting way. For example, it could be 
semi-conjugate to a substitution shift, a Sturmian shift, or an adding machine, see
\cite[\S 4.7]{Bruin2022_Top_and_erg_symb_dyn} and references therein.

In the last few decades, a great deal of work has clarified the relationship between the Cantor system arising
from unimodal maps and those classical systems like enumeration scales, Bratteli-Vershik systems, adding machines, see \cite{Bruin2003_Minimal_Cantor_systems_and_unimodal_maps,Bruin_et-all_Adding_Machines_and_wild_attractors,Alvin-Brucks_Adding_machines_kneading_maps_and_end_points, Alvin_Uniformly_recurrent_sequences_and_minimal_Cantor_omega-limit_sets}. These relations have been shown to be useful in the study of the metric properties of unimodal maps \cite{Lyubich-Milnor1993_Fib_unimodal,Bruin-et-all_1996_wild_cantor_attractors_exists, Cortez_Rivera-Letelier_Choquet_simplices, CorRL10}.

The relationship between the Cantor systems arising from unimodal maps restricted
to the $\omega$-limit set of its turning point and Bratteli-Vershik systems was first stated by 
Bruin \cite{Bruin2003_Minimal_Cantor_systems_and_unimodal_maps} and was obtained via the combinatorial data of $f$. More explicitly, the sequence of 
cutting times arising from $f$ gives rise to an enumeration scale that is semi-conjugate to a Bratteli-Vershik
system. Thus, despite being an explicit construction, it does not give immediate geometric information about the 
unimodal map.

In this paper, we provide a geometric construction of the $\omega$-limit set of the turning point for unimodal maps with Fibonacci-like combinatorics. Our construction makes explicit the geometric structure underlying the Bratteli–Vershik description and yields new information on both metric and smooth dynamical properties.
In this setting, the $\omega$-limit set of its turning point is a Cantor set and the action of the map 
restricted to it is minimal and uniquely ergodic. We construct a nested sequence of compact covers 
of the $\omega$-limit set, whose intersection equals the $\omega$-limit set. Each of these covers has a "tower structure", the containment relation between the base elements of these towers gives rise to a Bratteli diagram, and together with the action
of $f$ on these towers we recover the Bratteli-Vershik system described in \cite{Bruin2003_Minimal_Cantor_systems_and_unimodal_maps}, see also \cite[\S 5.4.5]{Bruin2022_Top_and_erg_symb_dyn}.

As applications of our construction, we are able to estimate the Hausdorff dimension of the attractor for Fibonacci-like maps in the family of symmetric tent maps. We also extend the results in \cite{Olivares-Vinales2023_Inv_meas_without_Lyap_exp} to maps with Fibonacci-like combinatorics.

We start by recalling some definitions to state our results.


\subsection{Covers and Bratteli diagrams} \label{subsec:B-v_thms}
In this paper a continuous map $f \colon [-1,1] \to [-1,1] $ is called \emph{unimodal} if $f(-1)=f(1)=-1$ and if there is 
$c \in (-1,1)$ such that $f|_{[-1,c)}$ is increasing and $f|_{(c,1]}$ is decreasing. We call the point $c$ the 
\emph{turning point of} $f$. We will assume that $f^2(c) < c < f(c)$; otherwise, the critical orbit is eventually attracted to an endpoint, and the situation is not of interest here. We will also assume that the map $f$ does not have wandering intervals. We denote the orbit of $x \in [-1,1]$ under $f$ by \[ \mathcal{O}_f(x) := \{ f^n(x) \colon n \geq 0 \}, \] we refer to its set of accumulation points as the
\emph{$\omega$-limit} set for $x$ under $f$ and we denote it by $\omega_f(x)$.

For each $n \in \N$, put $c_n \= f^n(c)$. Define the sequence of compact intervals $\{H_n\}_{n \ge 1}$ inductively by 
$H_1 = [c,c_1]$, and for each $n \ge 2$, by 

\begin{equation}
    \label{def:Hofbauer_tower}
    H_{n} \=  
    \begin{cases}
        f(H_{n-1}) & \text{ if } c \not\in H_{n-1}, \\
        [c_{n},c_1] & \text{ if } c \in H_{n-1}.
    \end{cases} 
\end{equation}
An integer $n \ge 1$ will be called a \emph{cutting time} if $c \in H_n$. This construction was introduced by Hofbauer \cite{Hofbauer-1980_The_top_entrpy_of_axx-1} and the collection $\{H_n\}_{n \ge 1}$ is usually called the \emph{Hofbauer tower} associated with $f$.
We will denote by $\{ S(k) \}_{k \ge 0}$ the
sequence of cutting times. From the assumption $c_2 < c< c_1$, it follows that $S(0)=1$ and $S(1) =2$.

We are interested in unimodal maps whose sequence of cutting times is given by 
\begin{equation}
    \label{eq:cut_times_formula}
    S_d(k) \=  
    \begin{cases}
        k+1 & \text{ if } k \le d, \\
        S_d(k-1) + S_d(k-d) & \text{ if } k > d,
    \end{cases} 
\end{equation}
for some integer $d \geq 2$. In this case, we will say that $f$ has \emph{Fibonacci-like combinatorics} or that 
$f$ is a Fibonacci-like unimodal map, and we will denote its sequence of cutting times by $\{ S_d(k)\}_{k \ge 0}$.
It is known that for a Fibonacci-like unimodal map $f$, we have that $\omega_{f}(c)$ is a Cantor set and the restriction
$f|_{\omega_{f}(c)}$ is minimal and uniquely ergodic; see \cite[Lemma 7, Lemma 8, Proposition 2]{Bruin2003_Minimal_Cantor_systems_and_unimodal_maps}, and \cite[Corollary 6.41]{Bruin2022_Top_and_erg_symb_dyn}.
Observe that if we define the map $Q_d(k) \= \max \{0, k-d\}$ for every integer $k \ge 1$, then 
$S_d(k) = S_d(k-1) + S_d(Q_d(k))$. 

For the rest of this section, let $f$ be a Fibonacci-like unimodal map with cutting times $\{ S_d(k) \}_{k \ge 0}$
for some integer $d \ge 2$.
We now state our first main result.
\begin{theorem}
    \label{theo:covers_of_the_omega_set}
    There exists a collection $\{M_{d,k}\}_{k \ge 0}$ of compact sets satisfying the following.
    \begin{itemize}
        \item[(i)]  For every integer $k \ge 0$, we have $\omega_f(c) \subset M_{d,k+1} \subset M_{d,k}$ and 
        $\omega_{f}(c) = \bigcap_{k \ge 0}M_{d,k}$.
        \item[(ii)] $M_{d,0} = [f^2(c),f(c)] $, and for $k = 1, \ldots, d-1$ there are $k+1$ compact intervals
        \[J_{1,k}, J_{d-k+1,k}, \ldots, J_{d,k}\] for which 
        \[ M_{d,k} = J_{1,k}\cup J_{d-k+1,k} \cup \cdots \cup J_{d,k}, \] and $J_{1,k}$ is the only interval containing 
        $c$.  
        \item[(iii)] For every integer $k \ge d-1$, there are $d$ compact intervals 
        \[ J_{1,k}, J_{2,k}, \ldots, J_{d,k} \]
        with  
        \[ M_{d,k} = \bigcup_{n=0}^{S_d(Q_d(k+1))-1}f^n(J_{1,k}) \quad \cup \quad \bigcup_{i=2}^d \quad \bigcup_{n=0}^{S_d(Q_d(k+i-d))-1}f^n(J_{i,k}), \] 
        $c \notin M_{d,k} \setminus J_{1,k}$, and 
        \[c \in f^{S_d(Q_d(k+1))}(J_{1,k}) \cap f^{S_d(Q_d(k+2-d))}(J_{2,k}) \cap \ldots \cap f^{S_d(Q_d(k))}(J_{d,k}).\]
        \item[(iv)] For every integer $k \ge 2d-1$, the intervals 
        \[ J_{1,k}, f(J_{1,k}), \ldots, f^{S_d(Q_d(k+1))-1}(J_{1,k}) \] are pairwise disjoint, and for 
        $i = 2, \ldots, d$, the intervals 
        \[ J_{i,k}, f(J_{i,k}), \ldots, f^{S_d(Q_d(k+i-d))-1}(J_{i,k}) \] are pairwise disjoint.
    \end{itemize}
\end{theorem}

The previous theorem tells us that each compact cover $M_{d,k}$ has a "tower" structure. In particular, for $k \ge d-1$, there are $d$
towers with base sets $J_{1,k}, \ldots, J_{d,k}$ and heights \[S_d(Q_d(k+1))-1, S_d(Q_d(k+2-d))-1, \ldots, S_d(Q_d(k))-1,\] respectively. 
Thus, if we have $x \in J_{i,k}$, its forward orbit will move along the tower, and once it reaches the "top level", it will be mapped by $f$ to a neighborhood of $c$.

During the construction of the covers, we will see that from one cover $M_{d,k}$ to the next one $M_{d,k+1}$, some towers
will remain the same, the tower containing the turning point will split into two towers, one containing the turning point (the central tower), which additionally stacks the last tower from the previous level on top of it, and the 
other that will have the same structure as the central tower of the previous level. 

An advantage of working with the covers described in \cref{theo:covers_of_the_omega_set} is that they not only control the topology of $\omega_f(c)$ but also provide a natural description of the unique $f$-invariant probability measure supported on $\omega_f(c)$. Let $\mu_d$ denote this measure. The following theorem gives an explicit combinatorial description of $\mu_d$ in terms of the covers provided by the previous theorem. In particular, it gives the measure of the $i$th tower at level $k$ for every integer $k \ge d-1$, and the measure of each floor at the $i$th tower at level $k$ for every integer $k \ge 2d-1$.

We introduce a notation for the "height" of the towers. It will be useful in certain cases when working with the tower dynamics, so it will be interchanged with the cutting time notation. For every integer $k \ge 1$, set 
\begin{equation}
    \label{eq:tower_height_def}
    h_{i,k} \=  
    \begin{cases}
        S_d(Q_d(k+1)) & \text{ if } i = 1, \\
        S_d(Q_d(k+i-d)) & \text{ if } \max \{d-k+1,2\} \le i \le d.
    \end{cases}
\end{equation}
We use the lower bound for $i$ given by $\max \{d-k+1,2\}$ because for $1 \le k \le d$, item $(ii)$ in \cref{theo:covers_of_the_omega_set} tells us that we have $k+1$ towers. 
\begin{theorem}
    \label{theo:measure_of_the_covers}
    Let $\beta_d$ be the largest real solution of the equation $x^d - x^{d-1} - 1 =0.$
    For every integer $k \ge d-1$ and every $1 \le i \le d$, we have
    \[ \mu_d \left( \bigcup_{n=0}^{h_{i,k}-1} f^n(J_{i,k})  \right) =
    \begin{cases}
        h_{i,k}\beta_d^{-k} & \text{ if } i = 1, \\
        h_{i,k}\beta_d^{-(k+i-1)} & \text{ otherwise. }
    \end{cases}
    \]
    Moreover, for every $k \ge 2d-1$ and every $0 \le n < h_{i,k}$, we have 
    \[
    \mu_d(f^{n}(J_{i,k})) = 
    \begin{cases}
        \beta_d^{-k} & \text{ if } i = 1, \\
        \beta_d^{-(k+i -1)} & \text{ otherwise.}
    \end{cases}
    \]
\end{theorem}

Now we will explain how to construct a graph using a containment relation on the connected components of each "tower".
Put $v_0 \= \{ [c_2,c_1]\}$, for $ k \ge 1$ and $\max \{d-k+1,2\} \le i \le d$, let 
\[v_k(i) \= \{ J_{i,k}, f(J_{i,k}), \ldots, f^{h_{i,k}-1}(J_{i,k}) \}.\] 
Put $V_0 = \{v_0\}$, and for $ k \ge 1$, put 
\[ V_k \= \{v_k(1)\} \cup \{ v_k(i) \colon \max\{d-k+1,2\} \leq i \leq d \}. \] 
Set $V^{(d)} \= \cup_{k \in \N_0} V_k$ as our set of vertices. For 
$k \in \N$, we have an edge $e$ connecting $v_k(i)$ with $v_{k+1}(j)$ if and only if there exists a set
$a \in v_k(i)$ and a set $b \in v_{k+1}(j)$ so that $b \subseteq a$. Thus, if $k \in \N_0$ and $\ell \ge 2$, then there 
is no edge between vertices from $V_k$ and $V_{k+\ell}$. Let $E_k$ be the set of edges from vertices in $V_{k-1}$ to 
vertices in $V_k$, and set $E^{(d)} \= \cup_{k \geq 1} E_k $.

\begin{theorem}
    \label{theo:B-V_system_from_the_cover}
    The infinite graph $B_d \= (V^{(d)},E^{(d)})$ is a Bratteli diagram. There exists a partial order $\le$ on $E^{(d)}$ for which 
    the ordered Bratteli diagram $(V^{(d)},E^{(d)}, \le)$ has a unique minimal path and $d$ distinct maximal paths, and the 
    Bratteli-Vershik system associated with it is semi-conjugate to the action of $f$ on $\omega_{f}(c)$.
\end{theorem}

The Bratteli diagram $B_d$ from the theorem above coincides with 
the ordered Bratteli diagram associated with the kneading map $Q_d$ that was introduced by Bruin \cite{Bruin2003_Minimal_Cantor_systems_and_unimodal_maps} after telescoping the first $d-1$ levels.


\subsection{Applications of the tower system}
One source of motivation to construct the collection of covers described above was to understand the topology of the Cantor set $\omega_f(c)$. In this regard, our constructions allow us to compute the Hausdorff dimension of $\omega_f(c)$ when $f$ is piecewise linear, in a direct and simple way. More precisely, for $a \in (1,2]$,
let $T_a(x) = a(1-|x|)-1$ be the piecewise linear unimodal map from $[-1,1]$ to itself. For this map, the turning point is $c = 0$. Let $a_d \in (1,2]$ be the 
unique parameter for which $T_{a_d}$ has Fibonacci-like combinatorics with cutting times $\{S_d(k)\}_{k\ge 0}$. 

\begin{theorem}
    \label{theo:HD}
    The Hausdorff dimension of the Cantor set $\omega_{T_{a_d}}(c)$ is zero.
\end{theorem}

This result reflects the extreme sparsity of the Cantor attractor in the Fibonacci-like regime.

\begin{remark}
    Using the same estimates as in \cite{Lyubich-Milnor1993_Fib_unimodal}, it is possible to prove that if $f$
    is a $C^2$-unimodal Fibonacci-like map with a non-degenerate critical point (and thus turning point) ($f''(c) \neq 0$), then the Cantor 
    set $\omega_f(c)$ has Hausdorff dimension zero.
\end{remark}


Another source of motivation for our construction was to study ergodic properties of maps with "cusps". We obtain, as a corollary of \cref{theo:covers_of_the_omega_set} and \cref{theo:measure_of_the_covers}, an extension of Theorem 1 in \cite{Olivares-Vinales2023_Inv_meas_without_Lyap_exp} for Fibonacci-like maps.

For every $A \subset [-1,1]$ and every $x \in [-1,1]$, we denote the 
\emph{distance from $x$ to $A$} by \[ \textup{dist}(x,A) := \inf \{ |x-y| \colon y \in A \}. \] 
For a smooth unimodal map $f$, we write $f'$ for its derivative. We say that the turning point $c$ of $f$ is a 
\emph{Lorenz-like singularity} if there exists $\ell^+$ and $\ell^-$ in $ (0,1)$, $L > 0$, and $\delta > 0$ such that the
following holds: For every $x \in (c, c + \delta)$,
\begin{equation}
  \label{eq:right_lo-like}
  \frac{1}{L |x - c|^{\ell^+}} \leq |f'(x)| \leq  \frac{L}{ |x - c|^{\ell^+}},
\end{equation}
and for every $x \in (c - \delta,c)$,
\begin{equation}
  \label{eq:left_lo-like}
  \frac{1}{L |x - c|^{\ell^-}} \leq |f'(x)| \leq \frac{L}{ |x - c|^{\ell^-}}.
\end{equation}
We call $\ell^+$ and $\ell^-$ the \emph{right and left order of} $c$, respectively.
For an interval map $f$, 
a point $\hat{c} \in [-1,1]$ is called a \emph{critical point} if $f'(\hat{c}) = 0$. We say that a critical point $\hat{c}$ is \emph{non-flat} if there
exist $\alpha^+ > 0$, $\alpha^- >0$, $M >0$, and $\delta >0$ such that the following holds: \newline \noindent
For every $x \in (\hat{c}, \hat{c} + \delta)$,
\begin{equation}
  \label{eq:right_crit_point}
  \left| \log \frac{|f'(x)|}{|x - \hat{c}|^{\alpha^+}} \right| \leq M,
\end{equation}
and for every $x \in (\hat{c} - \delta, \hat{c})$,
\begin{equation}
  \label{eq:left_crit_point}
  \left| \log \frac{|f'(x)|}{|x - \hat{c}|^{\alpha^-}} \right| \leq M.
\end{equation}
We call $\alpha^+$ and $\alpha^-$ the \emph{right and left order of} $\hat{c}$, respectively.
Let us denote by $\textup{Crit}(f)$ the set of critical points of $f$. If $f$ is a unimodal map with turning point~$c$, we will use the notation
$\mathcal{S}(f) := \textup{Crit}(f) \cup \{ c \}$.
Let us denote by $C^{\omega}$ the class of analytic maps. Here
we say that $f$ is a \emph{$C^{\omega}$-unimodal map} if it is of class $C^{\omega}$ outside $\mathcal{S}(f)$. 

For a probability measure $\mu$ on $[-1,1]$ that is invariant by $f$, we define the 
\emph{pushforward of $\mu$ by $f$} as \[ f_*\mu := \mu \circ f^{-1}.\] Denote by 
\[ \chi_{\mu}(f) := \int \log |f'| d\mu,\] its \emph{Lyapunov exponent} if the integral exists. Similarly, for every $x \in [-1,1]$ such that $\mathcal{O}_f(x) \cap \mathcal{S}(f) = \emptyset$, denote by
\[ \chi_f(x) := \lim_{n \to \infty} \frac{1}{n} \log|(f^n)'(x)|, \] the \emph{pointwise Lyapunov exponent of $f$ at $x$} if the limit exists.

Let $\mu_d$ be the unique measure that is ergodic, 
invariant by $T_{a_d}$, and supported on $\omega_{T_{a_d}}(c)$.

\begin{corollary}\label{theo:maps_with_cusps}
Let $h \colon [-1,1] \to [-1,1]$ be a homeomorphism of class $C^\omega$ on $[-1,1]~\setminus\{0\}$ with a unique non-flat critical point at $0$, and put $\tilde{\mu}_d := h_*\mu_d$.
Then, the $C^\omega$-unimodal map $f~:=~h~\circ~T_{a_d}~\circ~h^{-1} $ has a Lorenz-like singularity at $\tilde{c} := h(0)$ and satisfies the following properties:
\begin{enumerate}
\item $\chi_{\tilde{\mu}_d}(f) $ is not defined.
\vspace{0.3cm}
\item For $x \in \omega_f(\tilde{c})$, the pointwise Lyapunov exponent of $f$ at $x$ does not exist if $\mathcal{O}_f(x) \cap \mathcal{S}(f) = \emptyset$, and it is not defined if $\mathcal{O}_f(x) \cap \mathcal{S}(f) \neq \emptyset$.
\vspace{0.3cm}
\item $\log(\textup{dist}(\cdot,\mathcal{S}(f))) \notin L^1(\tilde{\mu}_d).$
\vspace{0.3cm}
\item $f$ has \emph{exponential recurrence} of the Lorenz-like singularity orbit, thus
    \begin{equation*}
        \limsup_{n \to \infty} \frac{ -\log |f^n(\tilde{c}) - \tilde{c}|}{n} \in (0, +\infty).
    \end{equation*}
  \end{enumerate}
\end{corollary}

The proof follows \cite{Olivares-Vinales2023_Inv_meas_without_Lyap_exp} with only minor modifications, and is omitted.


\subsection{Strategy and organization}

In \cref{sec:preliminaries}, we recall the definition of the kneading map associated with a unimodal map and its relation with the sequence of 
cutting times, as well as other related concepts. We present some properties of Fibonacci-like combinatorics.

In \cref{sec: closure of the orbit}, we make a detailed combinatorial description of the $\omega$-limit set
of the turning point for Fibonacci-like maps. We use this information to construct the sequence of covers described
in \cref{theo:covers_of_the_omega_set}, and prove that it satisfies all the properties mentioned there.

In \cref{sec:B-V_system}, after giving the definition of Bratteli-Vershik systems, we associate a graph to the 
collection of compact covers constructed in \cref{sec: closure of the orbit} and the containment of connected components relation,
leading to the proof of \cref{theo:B-V_system_from_the_cover}. Finally, we study the transition matrices associated with the Bratteli diagram, leading to the proof of \cref{theo:measure_of_the_covers}.

Finally, in \cref{sec:HD_proof}, we prove \cref{theo:HD}.

\subsection*{Acknowledgements}
The authors are grateful to Henk Bruin for helpful comments.
The first author was supported by the New Cornerstone Science Foundation through the New Cornerstone Investigator Program.
The second author was supported by the Institute for Basic Science (IBS-R029-C1).

\section{Preliminaries} \label{sec:preliminaries}

Throughout the paper, $\mathbb{N}$ denotes the set of positive integers, and $\mathbb{N}_{0}:=\mathbb{N} \cup \{0\}$. We also denote by $I$ the closed interval $\left[-1,1 \right] \subset \mathbb{R}$.

The interval $I$ is equipped with the distance induced by the absolute value $|\cdot|$ on $\mathbb{R}$. 
For an interval $J \subset I$, we denote its length by $|J|$.

For real numbers $a,b$, we put 
\[\left[ a,b \right]:=\left[ \min\{a,b\},\max\{a,b\}\right] \quad \text{ and } \quad (a,b):=(\min\{a,b\},\max\{a,b\}).\]

\subsection{Kneading theory of unimodal maps} \label{sec:kneading_theory}

Let $f:I \rightarrow I$ be a unimodal map with turning point $c \in (-1,1)$ and 
$f(-1)=f(1)=-1$. Put $c_i:=f^{i}(c)$ for $i \ge 1$ with $c_2 < c< c_1$. 
We are interested in the case when the orbit of the turning point is infinite, so we assume that $c$ is neither periodic nor eventually periodic.

In this section, we will describe the combinatorics of the turning point orbit.  
For a detailed exposition of the constructions presented below, 
we refer the reader to \cite[Chapter 3]{MevS11}, \cite[Chapter 6]{Brucks-Bruin2004_Topics_one_dimension}, and \cite[Section 3.6]{Bruin2022_Top_and_erg_symb_dyn}.

We define the \emph{itinerary} of a point $x \in I$ as the sequence $K(x) \= e_1(x)e_2(x)e_3(x)\ldots$,
where

\[e_n(x):= \begin{cases}
0 & \text{ if } f^n(x) < c \\
C & \text{ if } f^n(x) = c \\
1 & \text{ if } f^n(x) > c 
\end{cases}.\]
The itinerary of the turning point $c$ is called the \emph{kneading invariant} or 
\emph{kneading sequence} of $f$, and we denote it by $\e_f = (\e_k(f))_{k \ge 1}$. When the map is clear from the context, we omit the dependence on $f$ from the notation for the kneading sequence.
Since we assumed $c_2 < c < c_1$, we have 
$\e_1(f)\e_2(f) = 10$, and since we assumed the orbit of $c$ is not finite (in particular, $c$ is not periodic), we have $\e_i(f) \in \{ 0,1 \}$ for every $i \in \N$.

We can use the kneading sequences to define the cutting times introduced in \cref{subsec:B-v_thms}. Set $S(0) =1$, and for every $k \ge 1$ set
\[S(k):= \min \{ n > S(k-1) \mid c \in (c_n, c_{n-S(k-1)}) \} = \min\{n > S(k-1) \mid \e_n \ne \e_{n-S(k-1)}\}. \]
Thus, the cutting times allow us to split the kneading sequence into blocks
\begin{equation}
    \label{eq:cut_time}
    \Delta_k \= \e_{S(k-1)+1}\e_{S(k-1)+2} \ldots \e_{S(k)-1}\e_{S(k)} = \e_1\e_2 \ldots \e_{S(Q(k))-1}(1 - \e_{S(Q(k))}),
\end{equation}
for every $k \in \N$, where $Q(k) \in \N_0$ satisfies
\begin{equation}
    \label{eq:cut_time_kneading_map_relation}
    S(Q(k)) = S(k) - S(k-1).
\end{equation}

It can be proved that for every $k \in \N_0$ and every $S(k) < n \le S(k+1)$, we 
have $H_n \subset H_{n-S(k)}$, and they have a common boundary point, $c_{n-S(k)}$. Then
$ c \in H_{S(k+1)} \subset H_{S(k+1)-S(k)}$, and thus $S(Q(k+1)) = S(k+1)-S(k)$ is again a cutting time. 
Thus, \cref{eq:cut_time_kneading_map_relation} defines a function $Q:\mathbb{N}\to\mathbb{N}_0$ and
we call $Q$ the \textit{kneading map}.
Note that if either the kneading sequence, the cutting time, or the kneading map is given, the other two can be determined directly. 

We now specialize to the class of maps that will be studied throughout the paper.
\subsection{The Fibonacci-like maps}
We are interested in maps whose kneading map is given by 
\[Q_d(k)=\max\{0,k-d\},\] for a fixed $d \geq 2$. 
In this case, the cutting times are given by \cref{eq:cut_times_formula}.
For example, when $d=2$, $Q_2(k)=\max \{0,k-2\}$, and $S_2(k)$ follows the Fibonacci
recursion (with $S_2(0) = 1$, $S_2(1)=2$). Maps with this combinatorics are known as \emph{Fibonacci maps}. For $d \geq 2$,
we call these maps \emph{Fibonacci-like}. 

It is known that if $f$ is a Fibonacci-like unimodal map, then its turning point is recurrent, 
$\omega_{f}(c)$ is a Cantor set, and the restriction $f|_{\omega_{f}(c)}$ is minimal and uniquely ergodic; 
see \cite[Lemma 7, Lemma 8, Proposition 2]{Bruin2003_Minimal_Cantor_systems_and_unimodal_maps},
and \cite[Corollary 6.41]{Bruin2022_Top_and_erg_symb_dyn}.

Let $K_d = \e_1\e_2 \ldots$ be the kneading sequence determined by the kneading map $Q_d(k) = \max \{0,k-d \}$. 
Then the first $d+1$ symbols consist of the initial "$1$" followed by $d$ zeros,  
\begin{equation}
    \label{eq:prefix_kneading_invariant}
    \e_1\e_2\e_3 \ldots \e_{d+1} = 100 \ldots 0. 
\end{equation} 
For example, we have the following kneading sequences
\begin{align*}
    K_2 &= 100111011001010011100 \ldots \\
    K_3 &= 100011101100110001010 \ldots \\
    K_4 &= 100001110110011000110 \ldots 
\end{align*}

The following identity will be repeatedly used to compute tower heights and measure estimates.
\begin{lemma}
    \label{lem:sum_cut_times}
    For every $j \in \N_0$, we have 
    \[ \sum_{k=0}^{j} S_d(k) = S_d(j+d) - S_d(d-1). \]
\end{lemma}

\begin{proof}
    We prove the lemma by induction on $j$. For $j=0$, we get 
    \[ S_d(0) = S_d(d) - S_d(d-1) = d+1-d = 1. \] Thus, the result holds. Suppose that the 
    result holds for $j \in \N_0$. Then
    \begin{align*}
        \sum_{k=0}^{j+1}S_d(k) &= S_d(j+1) + \sum_{k=0}^{j}S_d(k) \\
            &= S_d(j+1) + S_d(j+d) -S_d(d-1) \\
            &= S_d(j+1+d) - S_d(d-1),
    \end{align*}
    where the second equality follows from the induction hypothesis, and the third equality 
    follows from \cref{eq:cut_times_formula}.
\end{proof}

The following lemma is a known result, see for example \cite[Exercise 6.1.6]{Brucks-Bruin2004_Topics_one_dimension}. 
We include a proof for the reader's convenience.

\begin{lemma}
    \label{lem:close_return_to_crit_value}
    For every $k \in \N$, we have $H_{S_d(k+1)+1} \subset H_{S_d(k)+1}$.
\end{lemma}

\begin{proof}
    For every $1 \le k \le d$,
    by \cref{eq:prefix_kneading_invariant}, $c_{S_d(k)} \in (-1,c)$, and by 
    \cref{eq:cut_times_formula} and since $f|_{(-1,c)}$ is increasing, 
    \[c_{S_d(k)} < c_{S_d(k+1)} = c_{S_d(k)+1}.\] By \cref{def:Hofbauer_tower}, 
    $H_{S_d(k)+1} = [c_{S_d(k)+1},c_1] = [c_{S_d(k+1)},c_1]$, and thus $H_{S_d(k+1)+1} \subset H_{S_d(k)+1}$.

    Now, for $k > d$, $Q_d(k) = k-d > 0$. So by \cref{eq:cut_time}, 
    $c_{S_d(k)+1} \in (c,c_1)$. Suppose, by contradiction, that the lemma does not hold; then there exists 
    $k > d$ such that $H_{S_d(k)+1} \subset H_{S_d(k+1)+1}$. Then, 
    \[ c \in f^{S_d(Q_d(k+1))}(H_{S_d(k)+1}) \subset f^{S_d(Q_d(k+1))}(H_{S_d(k+1)+1}) = 
    [c_{S_d(Q_d(k+1))}, c_{S_d(k+1)+S_d(Q_d(k+1))}],\] 
    by the definition of the Hofbauer tower and monotonicity on the corresponding branch.
    Since \[Q_d(k+1) = k+1-d < k+2-d = Q_d(k+2),\] and thus, $c_{S_d(Q_d(k+1))}$, and $c_{S_d(k+1)+S_d(Q_d(k+1))}$ are on the same side of $c$, contradicting \cref{eq:cut_time}. 
\end{proof}

We will use the notation $x \mapsto \tau(x)$ for the order-reversing involution
which satisfies $\tau(x) \neq x$, and  $f(\tau(x)) = f(x)$ on $I \setminus \{c\}$. 
For $x \in I \setminus \{c\}$, we put 
\[ x^+ \= \max \{ x, \tau(x) \} \hspace{0.5cm} \text{ and } \hspace{0.5cm}
x^- \= \min \{x, \tau(x) \}.\] Thus we have $x^- < c < x^+$, and $f(x^-) = f(x^+)$. 
Let $\|x-c\|$ denote the maximum of $|x-c|$ and $|\tau(x) -c|$.
Then, as a consequence of \cref{lem:close_return_to_crit_value}, we have 
\begin{equation}
    \label{eq:right_side_returns_to_c}
    |c_{S_d(k)}^+ - c| > |c_{S_d(k+1)}^+ - c|,
\end{equation}

\begin{equation}
    \label{eq:left_side_returns_to_c}
    |c_{S_d(k)}^- - c| > |c_{S_d(k+1)}^- - c|,
\end{equation}
and
\begin{equation}
    \label{eq:max_returns_to_c}
    \|c_{S_d(k)} - c\| > \|c_{S_d(k+1)} - c\|.
\end{equation}

\begin{remark}
    When the map under consideration is symmetric, for example, in the family 
    \[ f_{\lambda, \alpha}(x) = \lambda(1 - |x|^{\alpha}) - 1, \]
    inequalities \cref{eq:right_side_returns_to_c},
\cref{eq:left_side_returns_to_c}, and \cref{eq:max_returns_to_c} coincide. Thus
\begin{equation}
    \label{eq:close_retur_symm_tent_map}
    |c_{S_d(k)} - c| > |c_{S_d(k+1)} - c|
\end{equation}
for every $k \in \N$.
\end{remark}

\section{The set $\omega_f(c)$} \label{sec: closure of the orbit}

The purpose of this section is to prove \cref{theo:covers_of_the_omega_set}. Throughout this section, let
$f \colon I \longrightarrow I$ be a Fibonacci-like unimodal map with no wandering intervals, with turning point $c$, kneading map $Q_d$, and sequence of
cutting times $\{S_d(k)\}_{k \ge 0}$ for some integer $d \ge 2$. 
After an affine change of coordinates, we may
assume that $c = 0$. As before, for every $i \in \N$ put $c_i \= f^i(c)$. For any interval $J \subset I$, and any integer $n \ge 0$, write $J^n \= f^n(J)$.

For $k \geq 0$, let $I_{k} = I_k(d)$ be the smallest closed interval containing all of the points
$c_{S_d(j)}$ for every $j \ge k$. By iterating the block substitution rule \cref{eq:cut_time}, for every $m \in \N$ and $1 \le j \le d$,
we have
\[ \e_{S_d(md + j)} = 
\begin{cases}
    1- \e_{S_d(j)} & \text{ if } m \equiv 1  \pmod{2} \\
    \e_{S_d(j)} & \text{ if } m \equiv 0  \pmod{2}.
\end{cases}\]
Since $c_{S_d(j)}\to c$ 
and the parity rule for the relative position of $c_{S_d(md+j)}$, the extremal points of $I_k$ occur at indices $k$ and $k+ \ell$ with $\ell$ minimal such that $c_{S_d(k)}$ and $c_{S_d(k +\ell)}$ lie on opposite sides of $c$.
Then, by \cref{eq:prefix_kneading_invariant} and \cref{eq:max_returns_to_c}, if $k \equiv r \pmod{d}$,
\begin{equation}
    \label{eq:I_k_interval}
    I_k = 
    \begin{cases}
        [c_{S_d(k)}, c_{S_d(k+1)}] & \text{ if } r=0 \\
        [c_{S_d(k)},c_{S_d(k+d-r+1)}] & \text{ if } r \neq 0.
    \end{cases}    
\end{equation}

Alternatively, if we can write $k = j + nd$ with $j = 1, \ldots, d$ and $n \in \N_0$, then 
\begin{equation}
    \label{eq:I_k_short_def}
    I_k = [c_{S_d(k)}, c_{S_d(k-j+d+1)}].
\end{equation}

As an example, for $d = 2$ we get 
\[ I_k = 
\begin{cases}
    [c_{S_2(k)}, c_{S_2(k+1)}] & \text{ if } k \equiv 0 \pmod{2} \\
    [c_{S_2(k)}, c_{S_2(k+2)}] & \text{ if } k \equiv 1 \pmod{2}.
\end{cases} \]
And for $d = 3$ we get 
\[I_k=\begin{cases}
~[c_{S_3(k)},c_{S_3(k+1)}] & \text{ if } k\equiv 0 \pmod{3} \\
~[c_{S_3(k)},c_{S_3(k+2)}] & \text{ if } k\equiv 2 \pmod{3} \\
~[c_{S_3(k)},c_{S_3(k+3)}] & \text{ if } k\equiv 1 \pmod{3} 
\end{cases}.
\]

By the definition of $I_k$, we have that
for every $k \in \N$, $I_k^1 = H_{S_d(k)+1}$, where $H_n$ denotes the Hofbauer tower intervals defined in \cref{def:Hofbauer_tower}.
In particular, for every $k \geq d$ and $1 \le j < S_d(Q_d(k+1))$,  we have \[I_k^j = H_{S_d(k)+j}.\] 
By \cref{eq:cut_time} and \cref{eq:cut_times_formula}, we have $c \in [c_{S_d(k)+1},c_1]$ for every $0 \le k < d$.

\begin{lemma}\label{lem: injectivity}
    For every $k \ge d$ and every $j \in \{1,2,\ldots, S_d(Q_d(k+1))-1\}$, we have
    that $f^j$ is injective on $ \left [ c_{S_d(k)+1}, c_1 \right ]$.
    In particular, we have $ c \notin I_k^{j}= \left[ c_{S_d(k)+j}, c_{j} \right]$, and 
    \[ c \in I_{k}^{S_d(Q_d(k+1))}= \left[ c_{S_d(Q_d(k+1))}, c_{S_d(k+1)}\right].\]
\end{lemma}

\begin{proof}
Since $k \ge d$, $Q_d(k+1) = k+1-d$.
For $1 \le j < S_d(Q_d(k+1))$, \cref{eq:cut_time} implies that $c_{S_d(k)+j}$ and $c_j$ lie on the same side with respect to $c$. Hence $c \notin [c_{S_d(k)+j}, c_j]$, so $f^j$ is monotone (and therefore injective) on $[c_{S_d(k)+j}, c_j]$.
In particular, for $1 \le j < S_d(Q_d(k+1))$, 
\[I_k^j= f^{j-1}\left[ c_{S_d(k)+1},c_1\right]=\left[ c_{S_d(k)+j}, c_j\right],\]
as required.

Finally, from the above with $j = S_d(Q_d(k+1))$, we obtain 
\[ I_{k}^{S_d(Q_d(k+1))}=[ c_{S_d(k)+S_d(Q_d(k+1))}, c_{S_d(Q_d(k+1))}  ] = [ c_{S_d(k+1)}, c_{S_d(Q_d(k+1))} ]. \]
Then, it follows from \cref{eq:cut_time} that $c \in I_k^{S_d(Q_d(k+1))}$.
\end{proof}

\begin{lemma}\label{lem: distance with c}
For every $k \ge 1$, we have

\begin{equation}\label{eq: distance with c}
\|c_i - c\| > \|c_{S_d(k)}-c\|
\end{equation}
for all $0 < i < S_d(k)$.
\end{lemma}

\begin{proof} 
For $1 \le k \le d$, $S_d(k)=k+1$, and thus \cref{eq: distance with c} holds by 
\cref{eq:max_returns_to_c}. 

In the remaining cases, we argue by contradiction, using the nesting of the Hofbauer tower and monotonicity to pull back an assumed closer return to a contradiction with the kneading block rule.
Suppose the result is not true, and let $k \in \N$ be maximal so that \cref{eq: distance with c}
holds for every $n \leq k$. Thus $k \ge d$, and for every $1 \le i < S_d(k)$,
\[ \|c_i-c\| > \|c_{S_d(k)}-c\| > \|c_{S_d(k+1)}-c\|. \] In addition, there exists $S_d(k) < i' < S_d(k+1)$ such that 
\[ \|c_{i'} -c\| \le \|c_{S_d(k+1)}-c\|. \] Let $i'$ be the minimal satisfying the above. 
We can write $i' = S_d(k) + j$ with $1 \le j < S_d(Q_d(k+1))$, and thus
\begin{equation}
    \label{eq:contradiction}
    \|c_{S_d(k)+j} -c \| \le \|c_{S_d(k+1)} -c \|.
\end{equation}
In this case, by \cref{eq:max_returns_to_c} and \cref{lem: injectivity}, we must have 
\[ c \notin (c_{S_d(k)+j},c_j) \qquad \text{ and } \qquad c_{S_d(k+1)+j} \in (c_{S_d(k)+j},c_j). \]
Thus, 
\begin{equation}
    \label{eq:case_1_cls_rtrn_lem}
    \|c_{S_d(k)+j}-c\| > \|c_{S_d(k+1)+j}-c\| > \|c_j-c\|,
\end{equation}
or
\begin{equation}
    \label{eq:case_2_cls_rtrn_lem}
    \|c_j-c\| > \|c_{S_d(k+1)+j}-c\| > \|c_{S_d(k)+j}-c\|.
\end{equation}
If \cref{eq:case_1_cls_rtrn_lem} holds, then, by \cref{eq:contradiction}, and 
\cref{eq: distance with c}, we would have 
\[ \|c_{S_d(k+1)}-c\| \ge \|c_{S_d(k)+j}-c\| > \|c_{S_d(k+1)+j}-c\| > \|c_j-c\| > \|c_{S_d(k)}-c\|, \]
contradicting \cref{eq:max_returns_to_c}.

Now, if \cref{eq:case_2_cls_rtrn_lem} holds, we must have either 
\begin{equation}
    \label{eq:case_2.1_cls_rtrn_lem}
    \|c_j-c\| \ge \|c_{S_d(k+1)} - c\| > \|c_{S_d(k+1)+j} -c\| > \|c_{S_d(k)+j}-c\|,
\end{equation}
or
\begin{equation}
    \label{eq:case_2.2_cls_rtrn_lem}
    \|c_j-c\| > \|c_{S_d(k+1)+j} -c\| > \|c_{S_d(k+1)} - c\| > \|c_{S_d(k)+j}-c\|.
\end{equation}
If \cref{eq:case_2.1_cls_rtrn_lem} holds, we have 
\[ c_{j+1} < c_{S_d(k+1)+1} < c_{S_d(k+1)+j+1} < c_{S_d(k)+j+1} < c_1. \] Applying 
$f^{S_d(Q_d(k+1))-j-1}$ to the interval $(c_{j+1},c_1)$, we obtain
\begin{align*}
    c_{S_d(k+1)+S_d(Q_d(k+1))-j}, c_{S_d(k+1)+S_d(Q_d(k+1))}, c_{S_d(k+1)} &\in (c_{S_d(Q_d(k+1))}, c_{S_d(Q_d(k+1))-j}), \\
    c_{S_d(k+1)+S_d(Q_d(k+1))-j}, c_{S_d(k+1)+S_d(Q_d(k+1))} &\in (c_{S_d(Q_d(k+1))}, c_{S_d(k+1)}), \text{ and } \\
    c_{S_d(k+1)+S_d(Q_d(k+1))-j} &\in (c_{S_d(Q_d(k+1))},c_{S_d(k+1)+S_d(Q_d(k+1))}).
\end{align*}
By \cref{eq:cut_time} and \cref{lem: injectivity}, we have 
\begin{align*}
    c &\notin (c_{S_d(Q_d(k+1))}, c_{S_d(k+1)+S_d(Q_d(k+1))}), \text{ and }\\
    c &\notin (c_{S_d(k+1)+S_d(Q_d(k+1))-j}, c_{S_d(Q_d(k+1))-j}).
\end{align*}
Thus, $c \notin (c_{S_d(Q_d(k+1))}, c_{S_d(Q_d(k+1))-j})$. On the other hand, by \cref{eq:cut_time},
$c_{S_d(Q_d(k+1))}$ and $c_{S_d(k+1)}$ are on opposite sides with respect to $c$, and thus 
\[ c \in (c_{S_d(Q_d(k+1))}, c_{S_d(k+1)}) \subset (c_{S_d(Q_d(k+1))}, c_{S_d(Q_d(k+1))-j}), \] a contradiction.

Finally, suppose that \cref{eq:case_2.2_cls_rtrn_lem} holds. By 
\cref{eq:max_returns_to_c}, and since \cref{eq: distance with c} holds for $k$,
we have that 
\[ |c_j-c| > |c_{S_d(k)} -c| > |c_{S_d(k+1)+j} - c| \qquad \text{ or } \qquad 
|c_{S_d(k+1)+j} - c| > |c_{S_d(k)} -c| > |c_{S_d(k+1)}-c|. \]
In the former case, we have 
\[ c_{j+1} < c_{S_d(k)+1} < c_{S_d(k+1)+j+1} < c_{S_d(k+1)+1} < c_{S_d(k)+j+1} < c_1. \] 
By \cref{lem: injectivity}, $f^{S_d(Q_d(k+1))-j}$ is monotone on $(c_{j+1},c_{S_d(k)+j+1})$ and 
on $(c_{S_d(k)+1},c_1)$, so it maps the interval $(c_{j+1},c_1)$ homeomorphically onto the 
interval $(c_{S_d(Q_d(k+1))-j}, c_{S_d(Q_d(k+1))})$ with 
\[ \{c_{S_d(k+1)}, c_{S_d(k+1)+S_d(Q_d(k+1))-j,} c_{S_d(k+1)+S_d(Q_d(k+1))}, c_{S_d(k+1)-j} \} \subset  
(c_{S_d(Q_d(k+1))-j}, c_{S_d(Q_d(k+1))}),\] and 
\begin{equation*}
    \begin{split}
        |c_{S_d(k)+S_d(Q_d(k+1))}-c_{S_d(Q_d(k+1))-j}| &< |c_{S_d(k+1)+S_d(Q_d(k+1))-j} -c_{S_d(Q_d(k+1))-j}| \\ 
    &< |c_{S_d(k+1)+S_d(Q_d(k+1))} - c_{S_d(Q_d(k+1))-j}| \\
    &< |c_{S_d(k+1)-j} - c_{S_d(Q_d(k+1))-j}|.
    \end{split}
\end{equation*}
By \cref{eq:cut_time}, $c_{S_d(Q_d(k+1))-j}$ and $c_{S_d(k)+S_d(Q_d(k+1))-j} = c_{S_d(k+1)-j}$ are on the
same side with respect to $c$, and $c_{S_d(k+1)+S_d(Q_d(k+1))}$ and $c_{S_d(Q_d(k+1))}$ are on the same
side with respect to $c$. Since  \[c_{S_d(k+1)+S_d(Q_d(k+1))} \in (c_{S_d(Q_d(k+1))-j}, c_{S_d(k+1)-j}),\] we 
conclude that $c \notin (c_{S_d(Q_d(k+1))-j}, c_{S_d(Q_d(k+1))})$. On the other hand, 
$c_{S_d(k+1)}$ and $c_{S_d(Q_d(k+1))}$ are on opposite sides with respect to $c$, thus
\[ c  \in (c_{S_d(k+1)}, c_{S_d(Q_d(k+1))}) \subset (c_{S_d(Q_d(k+1))-j}, c_{S_d(Q_d(k+1))}), \] a contradiction.

In the latter case, we have 
\[ c_{j+1} < c_{S_d(k+1)+j+1} < c_{S_d(k)+1} < c_{S_d(k+1)+1} < c_{S_d(k)+j+1} < c_1. \]
By \cref{lem: injectivity}, $f^{S_d(Q_d(k+1))-j}$ maps the interval $(c_j^-,c)$ (or the interval $(c,c_j^+)$) 
homeomorphically onto $(c_{S_d(Q_d(k+1))}, c_{S_d(Q_d(k+1))-j})$, with 
\[ c_{S_d(k+1)}, c_{S_d(k+1)+S_d(Q_d(k+1))-j}, c_{S_d(k+1)+S_d(Q_d(k+1))}, c_{S_d(k)+S_d(Q_d(k+1))-j} \in  
(c_{S_d(Q_d(k+1))-j}, c_{S_d(Q_d(k+1))}),\] and 
\begin{align}
    \label{eq:later_case_order}
    \begin{split}
        |c_{S_d(k)+S_d(Q_d(k+1))}-c_{S_d(Q_d(k+1))}| &< |c_{S_d(k)+S_d(Q_d(k+1))-j} -c_{S_d(Q_d(k+1))}| \\ 
    &< |c_{S_d(k+1)+S_d(Q_d(k+1))-j} - c_{S_d(Q_d(k+1))}| \\
    &< |c_{S_d(k+1)} - c_{S_d(Q_d(k+1))}|.
    \end{split}
\end{align}
By \cref{eq:cut_time}, 
\[c \notin (c_{S_d(Q_d(k+1))}, c_{S_d(k+1)+S_d(Q_d(k+1))}) \cup (c_{S_d(k)+S_d(Q_d(k+1))-j}, c_{S_d(Q_d(k+1))-j}),\] 
and 
\[c \in (c_{S_d(Q_d(k+1))},c_{S_d(k+1)}).\] Thus, 
\begin{equation}
    \label{eq:c_contention}
    c \in (c_{S_d(k+1)+S_d(Q_d(k+1))},c_{S_d(k)+S_d(Q_d(k+1))-j}).
\end{equation}
By \cref{lem: injectivity}, $f^j$ maps the interval $(c_{S_d(k+1)+S_d(Q_d(k+1))-j},c_{S_d(Q_d(k+1))-j})$ homeomorphically onto
the interval $(c_{S_d(k+1)+S_d(Q_d(k+1))},c_{S_d(k)+S_d(Q_d(k+1))})$, and by our initial assumption, we have 
\[ |c_{S_d(k+1)}-c| < |c_{S_d(k)}-c| < |c_{S_d(Q_d(k+1))-j}-c|,\] then we must have 
\[ c_{S_d(k)+j} \in (c_{S_d(k+1)+j},c_{S_d(Q_d(k+1))}) \subset (c_{S_d(k+1)+S_d(Q_d(k+1))},c_{S_d(Q_d(k+1))}). \] 
By \cref{eq:later_case_order} and \cref{eq:c_contention}, this implies that 
\[ |c_{S_d(k+1)+j} -c| < |c_{S_d(k)+j} -c|, \] contradicting \cref{eq:case_2_cls_rtrn_lem}.
\end{proof}

\begin{lemma}
    \label{lem:I_k_disjoint}
    For every $k \ge 2d-1$, the intervals \[ I_k, \ldots, I_k^{S_d(Q_d(k+1))-1} \] are pairwise disjoint.
\end{lemma}

\begin{proof}
    We prove this lemma by contradiction. Suppose that there are $k \ge 2d-1$, and 
    integers $0 \le i < j < S_d(Q_d(k+1))$ so that \[ I_k^i \cap I_k^j \neq \emptyset. \] Since $j < S_d(Q_d(k+1))$, and $f(I_k) = f([c_{S_d(k)},c])$, \cref{lem: injectivity} implies that $f^j$ is monotone on $[c_{S_d(k)},c]$ and hence it maps $[c,c_{S_d(k)}]$ homeomorphically onto $I_k^j$. Then, we must have $[c,c_{S_d(k)}] \cap I_k^{j-i} \neq \emptyset$. Since $1 \le j-i < S_d(Q_d(k+1))$, \cref{lem: injectivity} implies that $c \notin I_k^{j-i}$ and thus one of the boundary points $c_{j-i}$, $c_{S_d(k) + j-i}$ belongs to $(c, c_{S_d(k)})$, and \cref{lem:close_return_to_crit_value} implies that $c_{j-i} \notin (c, c_{S_d(k)})$. Thus, we must have $c_{S_d(k) + j-i} \in (c,c_{S_d(k)})$. Using again that  $1 \le j-i < S_d(Q_d(k+1))$, the map $f^{S_d(Q_d(k+1))- (j-i)}$ maps $[c,c_{S_d(k)}] \cup I_k^{j-i}$ homeomorphically onto $I_k^{S_d(Q_d(k+1))-(j-i)} \cup I_k^{S_d(Q_d(k+1))}$, with \[c_{S_d(k+1)} \in I_k^{S_d(Q_d(k+1))-(j-i)} = [c_{S_d(Q_d(k+1))-(j-i)}, c_{S_d(k+1)-(j-i)}]. \] By \cref{lem: injectivity}, $c \in I_k^{S_d(Q_d(k+1))}$ and $c \notin I_k^{S_d(Q_d(k+1))-(j-i)}$, thus \[ |c_{S_d(Q_d(k+1))-(j-i)} - c| < |c_{S_d(k+1)} -c| \qquad \text{ or } \qquad |c_{S_d(k+1)-(j-i)} - c| < |c_{S_d(k+1)} -c|, \] contradicting \cref{lem:close_return_to_crit_value}. 
\end{proof}

\begin{remark}
    In the case $d=2,3$, the intervals are disjoint for every $k \ge 1$. For $d \ge 4$, if $k < 2d-1$, then $k+1-d < d$
    and thus $S_d(Q_d(k+1)) = Q_d(k+1)+1 < d+1$. Since $\e_2=\e_3 = \ldots = \e_{d+1} = 0$, this implies that 
    the intervals $I_k^2 , \ldots, I_k^{S_d(Q_d(k+1))-1}$ are to the left with respect to $c$, and by 
    \cref{lem: injectivity} $c \in I_k^{S_d(Q_d(k+1))}$. Now, if the intervals $I_k^2 , \ldots, I_k^{S_d(Q_d(k+1))-1}$ 
    are pairwise disjoint, then as $f$ is increasing to the left of $c$, we should have $c_i < c_{S_d(k)+i} < c_{i+1}$
    for every $i = 2, 3, \ldots, S_d(Q_d(k+1))-1$, but this implies that 
    \[ c \notin I_k^{S_d(Q_d(k+1))} = [c_{S_d(Q_d(k+1))}, c_{S_d(k+1)} ], \] 
    since $Q_d(k+1) < d$. 
\end{remark}

Now, we will construct the intervals $J_{1,k}, \ldots, J_{d,k}$ from \cref{theo:covers_of_the_omega_set}.  
Recall that, by \cref{eq:I_k_interval}, for $k \geq 0$, we have \[ I_k = [c_{S_d(k)},c_{S_d(k+j)}] ,\]
where $j \in \{ 1, \ldots, d \}$ is so that $k+j \equiv 1 \pmod{d}$ and $j$ is the smallest index such that 
$c_{S_d(k)}$ and $c_{S_d(k+j)}$ are on opposite sides with respect to $c$, see \cref{eq:I_k_short_def}.

Given $k \ge 1$, we have that the set \[I_{k+d}^{S_d(k)}=\left[c_{S_d(k)},c_{S_d(k)+S_d(k+d)} \right]\]
has a common boundary point with $I_k$. 

\begin{lemma}\label{lem: interval} For $k \ge 0$, we have
\begin{equation}
    \label{lem: interval eq 1}
    I_{k+d}^{S_d(k)} \subset I_{k},
\end{equation}
and for $k \ge d-1$,
\begin{equation}
    \label{lem: interval eq 2}
    I_{k+d}^{S_d(k)} \cap I_{k+1}= \emptyset
\end{equation}
\end{lemma}

\begin{proof}
First, we prove \cref{lem: interval eq 1}.
By \cref{lem: injectivity}, for every $j \in \{1,2,\ldots,S_d(k+1)-1\}$, we have
\[c \notin I_{k+d}^j=\left[c_j,c_{S_d(k+d)+j} \right].\]
Since $1 \le S_d(k) < S_d(k+1)$, taking $j=S_d(k)$ gives
\[c \notin I_{k+d}^{S_d(k)}=\left[c_{S_d(k)},c_{S_d(k+d)+S_d(k)} \right].\]
Then, since $I_{k+d}^{S_d(k)}$ and $I_k$ have $c_{S_d(k)}$ as a common boundary point, we have either 
\[ I_{k+d}^{S_d(k)} \cap I_k=\{c_{S_d(k)}\} \hspace{0.5cm} \text{ or } \hspace{0.5cm} I_{k+d}^{S_d(k)} \subset I_k.\]
If $I_{k+d}^{S_d(k)} \cap I_k=\{c_{S_d(k)}\}$, then
\[I_{k+d}^{S_d(k)+1} \cap I_k^1= \left[c_{S_d(k+d)+S_d(k)+1},c_{S_d(k)+1} \right] \cap 
\left[ c_{S_d(k)+1},c_1\right]=\{c_{S_d(k)+1}\}\]
and
\[c_{S_d(k+d)+S_d(k)+1} < c_{S_d(k)+1} < c_1.\]
By \cref{lem: injectivity}, 
\begin{align}\label{cin}
c \in I_{k+d}^{S_d(k+1)} \quad \text{and} \quad c \in I_{k}^{S_d(Q_d(k+1))},
\end{align}
and, since $S_d(k+1)=S_d(k)+S_d(Q_d(k+1))$,
\[c \notin I_{k+d}^{S_d(k)+j} \cup I_k^j \quad \text{for} \; j \in \{1,2,\ldots,S_d(Q_d(k+1))-1\}.\]
In particular, 
\[I_{k+d}^{S_d(k+1)} \cap I_k^{S_d(Q_d(k+1))}=\{c_{S_d(k+1)}\}\neq \{c\},\]
contradicting \cref{cin}. Thus we must have $I_{k+d}^{S_d(k)} \subset I_k$.

Finally, we prove \cref{lem: interval eq 2}. Suppose the statement of \cref{lem: interval eq 2} is not true. By \cref{eq:I_k_interval}, for every $k \ge d-1$, 
$I_k$ and $I_{k+1}$ have a common boundary point. If $k \equiv\ell \pmod{d}$, then
\[\partial I_k \cap \partial I_{k+1}=
\begin{cases}
    c_{S_d(k+1+d-\ell)}  & \text{ if } \ell \neq 0 \\
    c_{S_d(k+1)} & \text{ if } \ell = 0
\end{cases}.\]
If $k \equiv\ell \pmod{d}$ with $\ell \neq 0$, then
\[ \partial I_k \cap \partial I_{k+1}= \{ c_{S_d(k+1+d-\ell)} \} \hspace{0.5cm} \text{ and } \hspace{0.5cm}
\partial I_k \cap \partial I_{k+d}^{S_d(k)} = \{ c_{S_d(k)} \}, \] so we must have 
\[ |c_{S_d(k+1)} - c| > | c_{S_d(k+d)+S_d(k)} -c|. \] This implies $c_{S_d(k+1)+1} < c_{S_d(k+d)+S_d(k)+1} < c_1$, 
and thus
\begin{equation}
    \label{eq:cont_contradiction}
    I_{k+d}^{S_d(k)+1} \subset I_{k+1}^1.
\end{equation}
Now, by Lemma \ref{lem: injectivity}, for $1 \le j < S_d(Q_d(k+1))$, \[ c \not\in I_{k+d}^{S_d(k)+j} \quad 
\text{ and } \quad c \in I_{k+d}^{S_d(k+1)} = I_{k+d}^{S_d(k)+S_d(Q_d(k+1))}. \]  On the other hand, for 
$1 \le j < S_d(k+2-d)$, \[c \not\in I_{k+1}^j.\] By \cref{eq:cont_contradiction}, we have 
\[ c \in I_{k+d}^{S_d(k)+S_d(Q_d(k+1))} \subset I_{k+1}^{S_d(Q_d(k+1))}, \] so this leads to a contradiction since 
$S_d(Q_d(k+1)) < S_d(Q_d(k+2))$.

If $k\equiv 0 \pmod{d}$, by the same argument as above, then we must have 
\[ |c_{S_d(k+1+d)} - c | > |c_{S_d(k+d)+S_d(k)} - c|. \] Since $S_d(k+d)+S_d(k) < S_d(k+1+d)$, this contradicts 
\cref{lem: distance with c}.
\end{proof}

For $ 1 \le k < d$, the intervals $I_{k+d}^{S_d(k)}$ and $I_{k+1}$ have non-empty intersection. Indeed, since $I_{k+d}^{S_d(k)}$ is to the left of $c$, and since $f(I_{k+d}^{S_d(k)}) = I_{k+d}^{S_d(k+1)} \ni c$, we must have that the right boundary point of $I_{k+d}^{S_d(k)}$, that is $c_{S_d(k) + S_d(k+d)}$, must be located to the right of $c_{S_d(k+1)}$ that is the left boundary point of $I_{k+1}$.

Our next task is to construct $d$ intervals that will serve as the basis of our tower system, with only one of them containing the turning point $c$. 
If we start with $I_0 = [c_2,c_1]$, by \cref{lem: interval}, we have that the intervals $I_1$ and $I_d^{S_d(0)}$ are contained in $I_0$. In the next step, we have that $I_2$ and $I_{d+1}^{S_d(1)}$ are contained in $I_1$, the interval $I_2$ contains the turning point $c$. We can keep this procedure, and at each step we obtain two new intervals contained in the interval in the previous step containing $c$. 

Also, observe that for the first $d-1$ steps of this construction, the image of each interval we have contains the turning point $c$. When we reach the step $d-1$, we are left with $d$ intervals (excluding the initial interval $I_0$) \[ I_{d-1}, I_d^{S_d(0)}, I_{d+1}^{S_d(1)}, \ldots, I_{2d-2}^{S_d(d-2)}. \]  
At the next step, we generate two new intervals $I_d$ containing the turning point $c$, and the interval $I_{2d-1}^{S_d(d-1)}$, and we can observe that one of the intervals generated at the first step, $I_d^{S_d(0)}$, belongs to the orbit of $I_d$. Thus, at this step, we can always generate two new intervals, and one of our previous intervals gets "absorbed" in the orbit of the new interval containing the turning point.

This construction can be carried out inductively, leading to the definition of the following "base" 
intervals: For $k \ge 1$, put 
\begin{equation}
    \label{eq:def_central_tower}
    J_{1,k} \= I_k,
\end{equation}
and for $\max\{2,d-k+1\} \leq j \le d$, put
\begin{equation}
    \label{eq:def_J_ik}
    J_{j,k}=I_{k-1+j}^{S_d(Q_d(k-1+j))}=\left[c_{S_d(k-1+j)+S_d(Q_d(k-1+j))},c_{S_d(Q_d(k-1+j))} \right]
\end{equation}
whenever $k>d-j$.
Observe that, by definition, for every $k \in \N$ and every $3 \le j \le d$ for which $k > d-j$,
\begin{equation}
    \label{eq:J_i_k_to_next_level}
    J_{j,k} = J_{j-1,k+1}.
\end{equation}
Also, since $S_d(k+j+1)=S_d(k+j)+S_d(Q_d(k+j+1))=S_d(k+j)+S_d(Q_d(k+j))+S_d(Q_d(k+j+1-d))$,
\begin{align*}
f^{S_d(Q_d(k+j+1-d))}(J_{j,k+1})&=I_{k+j}^{S_d(Q_d(k+j))+S_d(Q_d(k+j+1-d))}\\
&=\left[ c_{S_d(k+j)+S_d(Q_d(k+j))+S_d(Q_d(k+j+1-d))},c_{S_d(Q_d(k+j))+S_d(Q_d(k+j+1-d))}\right]\\
&=\left[c_{S_d(k+j+1)},c_{S_d(Q_d(k+j+1))} \right].
\end{align*}
Thus, $c \notin J_{j,k+1}^i$ for $1 \le i < S_d(Q_d(k+1+j-d))$ whenever $k \ge 2d-1$,
and by \cref{lem:I_k_disjoint}, for every $k \ge 2d-1$, the intervals 
\[ J_{i,k}, J_{i,k}^1, \ldots, J_{i,k}^{h_{i,k}-1} \] 
are pairwise disjoint. Here, $h_{i,k}$ is defined in \cref{eq:tower_height_def}.

We use these intervals to define a sequence of compact sets. Let $M_{d,0} \= I_0$, for $k \ge 1$ and $\max \{d-k+1,2\} \le i \le d$ let $h_{i,k}$ be as in \cref{eq:tower_height_def}, and put 

\begin{equation}\label{eq: union}
        M_{d,k} \= \bigcup_{n=0}^{h_{1,k}-1}J_{1,k}^n \quad \cup  \bigcup_{i=\max \{d-k+1,2\}}^{d} \bigcup_{n = 0}^{h_{i,k}-1}J_{i,k}^n.
    \end{equation}

The following are the first $d +1$ sets $M_{d,k}$. 
\begin{align}
    \label{eq:d_levels_of_M_d,k}
    \begin{split}
        M_{d,0}&:=I_0 \\
        M_{d,1}&=J_{1,1} \cup J_{d,1} \\
        M_{d,2}&=J_{1,2} \cup J_{d,2} \cup J_{d-1,2} \\
        M_{d,3}&=J_{1,3} \cup J_{d,3} \cup J_{d-1,3} \cup J_{d-2,3} \\
        & \vdots \\
        M_{d,d-1}&= J_{1,d-1} \cup J_{d,d-1} \cup J_{d-1,d-1} \cup \ldots \cup J_{2,d-1} \\
        M_{d,d}&= J_{1,d} \cup J_{1,d}^1 \cup J_{d,d} \cup J_{d-1,d} \cup \ldots \cup J_{2,d}.
    \end{split}
\end{align}

For each $k \in \N$, we call the \emph{central tower} of the set $M_{d,k}$ the collection
\[ \cT_{d,k}(1) \=  \{ J_{1,k}, \ldots ,  J_{1,k}^{h_{1,k}-1} \}, \] 
and the interval $J_{1,k}$ will be called the \emph{base} of the central tower. For $\max \{d-k+1,2\} \leq i \le d$, 
we call the \emph{$i$th tower} of $M_{d,k}$ the collection 
\[ \cT_{d,k}(i) \= \{ J_{i,k}, \ldots ,J_{i,k}^{h_{i,k}-1} \}, \] 
and the interval $J_{i,k}$ will be called the \emph{base} of the $i$th tower. The \emph{height} of each tower will 
be the number of sets in it.

Thus, for $1 \le k \le d-1$, the set $M_{d,k}$ consists of $k + 1$ towers, each of them of height one. For $k \ge d$, the 
set $M_{d,k}$ consists of $d$ towers, its central tower has 
height $h_{1,k} = S_d(Q_d(k+1))$, and for $2 \le i \le d$, the $i$th tower has height $h_{i,k} = S_d(Q_d(k+i-d))$.

By \cref{lem:I_k_disjoint}, for $k \ge 2d-1$ and every $1 \le i \le d$, each tower $\cT_{d,k}(i)$ is a collection of 
pairwise disjoint intervals. Intervals from different towers need not be disjoint.
Indeed, each top level is mapped homeomorphically by $f$ onto a neighborhood of $c$, hence each must contain a point $x \in f^{-1}(c)$. Since $c$ has only two preimages, at least two top levels contain the same preimage, and therefore intersect. 

\begin{proof}[Proof of \cref{theo:covers_of_the_omega_set}]
Let $M_{d,k}$ be given by \cref{eq: union}. We start proving item $(i)$.
For every $k \in \N_0$, each $M_{d,k}$ is compact by construction, also
by \cref{lem: interval} and \cref{eq:J_i_k_to_next_level}, we have that $M_{d,k+1} \subset M_{d,k}$. 
Now we prove that $\omega_f(c) \subset M_{d,k}$ for every $k \in \N_0$.
By construction, for every $k \in \N_0$, each boundary point of $M_{d,k}$ is also a boundary point of $M_{d,k+1}$. 
Suppose that there exists a $k \in \N_0$ and $n \in \N$ so that $c_n \notin M_{d,k}$. 
Let $i \in \N$ be the largest for
which $S_d(i) \le n < S_d(i+1)$, and write $n = S_d(i) + m$. If $1 \le i < d$, then $n = S_d(i)$ and $c_n$ is a boundary point
of $I_n$. If $i \ge d$, by definition of the intervals $I_k$ and since 
$\partial M_{d,k} \subset \partial M_{d,k+1}$, we have that $\{ c_{S_d(j)} \}_{j \in \N_0} \subset M_{d,k} $ for every 
$k \in \N_0$. So $1 \le m < S_d(Q_d(i+1))$. If $1 \le m < S_d(Q_d(i))$, then $c_{S_d(i)+m} \in I_{i-1}^m \supset I_i^m$, and thus
is contained in $M_{d,k}$ for every $k \in \N_0$. Now, if $S_d(Q_d(i)) \le m < S_d(Q_d(i+1))$, we can write $ i = k'-1+j$ for 
some $2 \le j \le d$. Then \[ c_{S_d(i)+m} = c_{S_d(k'-1+j)+m} \in I_{k'-1+j}^{S_d(Q_d(k'-1+j))+m'}, \] where 
$m' = m - S_d(Q_d(k'-1+j)) = m - S_d(Q_d(i))$. Thus, $c_{S_d(i)+m} \in J_{j,k'}^{m'} \subset M_{d,k'}$. Since 
$M_{d,k'} \subseteq M_{d,k}$, we get $c_{S_d(i)+m} \in M_{d,k}$, contradicting our assumption.

Thus, we obtain \[ M_{d,0} \supset M_{d,1} \supset M_{d,2} \supset \ldots \supset \cO_f(c). \] 
Conversely, if $x \notin \ov{\cO_f(c)}$, then $x$ lies in the interior of a complementary interval to $\ov{\cO_f(c)}$, which would be  wandering unless it is eventually excluded from the nested covers, hence $x \notin \cap_{k \ge 1} M_{d,k}$. 
We conclude that 
\[ \bigcap_{k \in \N_0} M_{d,k} = \ov{\cO_f(c)} = \omega_{f}(c), \]
where the last equality is given by the fact that $c$ is recurrent and non-periodic. This concludes the proof of item $(i)$.

The proof of items $(ii)$ and $(iii)$ is a direct consequence of \cref{eq:def_J_ik}, \cref{eq: union}, and \cref{lem:I_k_disjoint}.

Finally, the proof of item $(iv)$ is a direct consequence of \cref{lem:I_k_disjoint}.
\end{proof}

\section{Bratteli-Vershik system associated with the collection $\{M_{d,k}\}_{k \geq 0}$} \label{sec:B-V_system}

In this section, we will associate a Bratteli diagram with the collection 
$\{ M_{d,k} \}_{k \ge 0}$. We will use the "tower" structure of each of the sets $M_{d,k}$ and how 
it changes from one level to the next 
(compare with \cite[Section 4]{Bruin2003_Minimal_Cantor_systems_and_unimodal_maps}, 
and \cite[Section 5.4.5]{Bruin2022_Top_and_erg_symb_dyn}).
We start by recalling the concepts of Bratteli diagrams and Bratteli-Vershik systems.

\subsection{Bratteli diagrams.}\label{subsec: Bratteli diagrams} 

A \emph{Bratteli diagram} is an infinite graph $B=(V,E)$, such that the vertex set $V$ and the edge set 
$E$ can be partitioned into a sequence of disjoint, finite, and  non-empty sets 
\[ V = V_0 \cup V_1 \cup \cdots \quad \text{ and } \quad E = E_1 \cup E_2 \cup \cdots \] satisfying the following:

\begin{itemize}
    \item[(i)] $V_0 = \{v_0\}$ consists of a single vertex;
    \item[(ii)] there exist functions $\phi \colon E \rightarrow V \times V$ (incidence map),
    $t \colon E \rightarrow V$ (target map), and $s \colon E \rightarrow V$ (source map)
    such that for every $i \in \N$, $\phi(E_i) \subset V_{i-1}\times V_i$, $t(E_i) = V_{i}$ and $s(E_i) = V_{i-1}$, 
    and $\phi(e) = (s(e),t(e))$ for every $e \in E$; 
    \item[(iii)] $t^{-1}(v) \neq \emptyset$ and $s^{-1}(v') \neq \emptyset$ for every $v \in V \setminus V_0$ and every 
    $v' \in V$.
\end{itemize}
Condition (ii) says that for every $i \in \N$, each edge $e \in E_i$ connects a vertex in $V_{i-1}$ with a vertex
in $V_i$, so we can associate each edge of $B$ with an ordered pair in $V \times V$. Condition (iii) says that for every 
$i \in \N_0$, each vertex in $V_i$ is connected to at least one vertex
in $V_{i+1}$ and each vertex in $V_{i+1}$ is connected to at least one vertex in $V_i$.

An \emph{infinite path} in $(V,E)$ is an infinite sequence $(e_j)_{j \ge 1} = e_1 e_2 e_3 \ldots$ such that 
$e_j \in E$ and $t(e_j) = s(e_{j+1})$ for every $j \in \N$. A \emph{finite path} is a string $e_1 e_2 \ldots e_k$ 
such that each $e_j \in E$ and $t(e_j) = s(e_{j+1})$ for $1 \le j < k$. By condition (iii), for every vertex $v \in V$
different from $v_0$, there exists at least one finite path starting from $v_0$ and arriving at $v$.

An \emph{ordered Bratteli diagram} $(V,E, \le)$  is a Bratteli diagram $(V,E)$ together with a linear order $\le$
defined on every set $t^{-1}(v)$ with $v \in V\setminus V_0$.  For every $j \in \N$, the partial order $\le$ induces an order in the set of finite paths from $v_0$ 
to $V_j$ in the following way: \[ e_1e_2 \ldots e_j < f_1f_2 \ldots f_j \] if and only if there exists a 
$j_0 \in \{ 1,2, \ldots j\}$ such that $e_{j_0} < f_{j_0}$ and for every $ j_0 < k \le j$, we have $e_k = f_k$. 
We use the same notation $\le$ to denote this partial order in the set of finite paths.

We say that an edge $e \in E$ is \emph{maximal} (respectively \emph{minimal}) if it is maximal (respectively minimal)
with respect to the order $\le$ on the set $t^{-1}(t(e))$. We call a finite path $e_1 \ldots e_k$ \emph{maximal}
(respectively \emph{minimal}) if $e_i$ is maximal (respectively minimal) for every $1 \le i \le k$. Note that, for 
$v \in V \setminus V_0$, there is a unique maximal (respectively minimal) path from $v_0$ to $v$.

\subsection{Bratteli-Vershik systems.} 
Fix an ordered Bratteli diagram $B = (V,E, \le)$. Denote by $X_B$ the set of infinite paths starting at $v_0$, thus 
\[X_B \= \{ (e_i)_{i \ge1} \in E^{\N} \colon s(e_1) = v_0 \text{ and } s(e_{i+1}) = t(e_{i}) \text{ for every } i \in \N \}.  \]
Given a 
finite path $e_1 \ldots e_k$ starting at $v_0$, we denote by $[e_1 \ldots e_k]$ the set of infinite paths 
$(f_n)_{n \geq 1}$ in $X_B$ such that for all $1 \le i \le k$, we have $f_i = e_i$. We endow $X_B$ with the topology
generated by the sets $[e_1 \ldots e_k]$. Then each of these sets is clopen, so $X_B$ becomes a compact Hausdorff 
space with a countable basis of clopen sets. 

We denote by $X_B^{\max}$ (respectively $X_B^{\min}$) the subset of $X_B$ consisting of all infinite paths 
$(e_i)_{i \ge 1}$ for which each of the $e_i$ is a maximal (respectively minimal) edge. We see that each of 
these sets is non-empty. We only consider the case when the set $X_B^{\min}$ is 
reduced to a single sequence that we denote by $x_{\min}$, as is the case in our setting.

The \emph{Vershik adic transformation} (also called the \emph{Vershik map}) $V_B \colon X_B \longrightarrow X_B$ is defined as follows:
\begin{itemize}
    \item $V_B^{-1}(x_{\min}) = X_B^{\max}$.
    \item Given $x = (x_i)_{i\geq 1} \in X_B \setminus X_B^{\max}$, let $j \geq 1$ be the smallest integer such that
    $x_j$ is not maximal. Let $e_j$ be the successor of $x_j$, and $e_1 \ldots e_{j-1}$ be the unique minimal path from
    $v_0$ to $s(e_j)$. Then we put \[ V_B(x) = e_1 \ldots e_{j-1} e_j x_{j+1} x_{j+2} \ldots. \]
\end{itemize}
This map is continuous, onto, and invertible except at $ x_{ \min }$, which may have several preimages.


\subsection{Incidence matrices} \label{subsec:Matrix_BV_syst}
We fix an ordered Bratteli diagram $B = (V,E,\le)$ having a unique minimal infinite path, and consider the Vershik map 
$V_B \colon X_B \longrightarrow X_B$ defined in the previous subsection.

Given non-empty finite sets $V, V'$, denote by $\cM_{V,V'}$ the set of matrices whose entries are real and indexed by $V \times V'$. For a matrix $A \in \cM_{V,V'}$, we 
denote by $A^t$ the transpose of $A$, and for $(v,v') \in V \times V'$, we denote by $A(v,v')$ the corresponding entry of $A$. 

For $j \geq 1$ and $v \in V_j$, we denote by $N_j(v)$ the number of paths starting at $v_0$ and arriving to $v$ and put
$\vec{N_j} = (N_j(v))_{v \in V_j} \in \R^{V_j}$. To each $E_j$, we assign an incidence matrix 
$F_j \in \cM_{V_{j-1}\times V_j}$ such that for each $v \in V_{j-1}$ and $v' \in V_j$ the entry $F_j(v,v')$ is equal to
the number of edges from $v$ to $v'$. It follows by induction that $\vec{N}_j = \vec{N}_{j-1}F_j$.


\subsection{The Bratteli-Vershik system associated with the collection $\{M_{d,k}\}_{k \ge 0}$}\label{subsection:BV_system_associated_to_Mdk}

We start constructing a Bratteli diagram associated with the collection $\{ M_{d,k} \}_{k \ge 0}$. The idea is as 
follows: each vertex in the set $V_k$ corresponds to one of the towers of $M_{d,k}$.
Then, given two vertices $a$ and $b$, we will have an edge $e$ with $s(e)=a$ and
$t(e)=b$ if $a$ and $b$ correspond to towers in consecutive sets of $\{M_{d,k}\}_{k \ge 0}$ and
the tower associated with the vertex $b$ is "contained" in the tower associated with the vertex $a$.

Now, we present a precise construction.
For every integer $k \ge 1$, put $v_k(1) \= \cT_{d,k}(1)$, and for $\max\{d-k+1,2\} \le i \le d$, put 
$v_k(i) \= \cT_{d,k}(i)$.
Put $V_0= \{v_0\} \= \{I_0\}$, for $ 1 \le k < d-1$, \[V_k:=\{v_k(1),v_k(d-k+1),\ldots,v_k(d)\},\] and for $k \ge d-1$, \[V_k:=\{v_k(1),v_k(2),\ldots,v_k(d)\}.\] 
And thus, for $0 \le i \le d-1$, each set $V_i$ consists of $i+1$ vertices, and for $i \ge d-1$, each set $V_i$ consists
of $d$ vertices. So $V = \cup_{k \geq 0} V_k$ is our set of vertices.

Now, for each integer $k \geq 1$, we will define a collection of edges $E_k$ connecting vertices from $V_{k-1}$ with vertices in $V_{k}$. The edges represent how each tower in $M_{d,k-1}$ changes with respect to the towers in $M_{d,k}$. 

By \cref{lem: interval}, for every integer $k \ge 1$, the interval $J_{1,k-1}$ splits into two intervals $J_{1,k}$ and $J_{d,k}$. This means that for each $k \in \N$, we have two edges $e_k(1,1),e_k(1,d) \in E_k$ with $s(e_k(1,1)) = s(e_k(1,d)) =v_{k-1}(1)$, $t(e_k(1,1)) =v_k(1)$, and  $t(e_k(1,d))=v_k(d)$. Also, by \cref{eq:J_i_k_to_next_level}, for every $k \in \N$, and each $3 \le j <d$ for which $k>d-j$, there is a unique edge $e_k(j,j-1)$ with $s(e_k(j,j-1))=v_{k-1}(j)$ and $t(e_k(j,j-1))=v_k(j-1)$. For each of these edges, we set $\phi(e_k(i,j)) = (v_{k-1}(i),v_k(j))$. 

Thus, we have that for $k \in \N$, $v_k(1)$ has two "outgoing" edges, any other vertex $v_k(i)$ has a single outgoing edge. For $k \le d-1$, each vertex has a unique incoming edge,
and for $k \ge d$
the vertex $v_k(1)$ has two "incoming" edges, any other vertex in $V \setminus V_0$ has one "outgoing" and one
"incoming" edge, see \cref{fig:Bratelli_diagram_first_d_levels}.

Thus, for every $i \in \N$, we can associate, without ambiguity, to each $e \in E_k$ the uniquely defined pair 
$(s(e), t(e)) \in V_{k-1} \times V_k$. It is not hard to check that the graph (V,E) constructed above is 
a Bratteli diagram.
We will denote the ordered Bratteli diagram associated with the collection $\{M_{d,k}\}_{k \ge 0}$ constructed above by $B_d$.

We now have the Bratteli diagram, and the next step is to understand the transition matrices and their product.
The following lemma describes the incidence matrices of the Bratteli diagram $B_d$. Its proof is direct from the definition 
of the Bratteli diagram, since for $k \ge d$, this one becomes stationary. We will denote by $F_{k,d}$ the $k$th incidence matrix of the Bratteli diagram $B_d$.
\begin{lemma}
    \label{lem:inc_mtx}
    For every $1 \le k < d$, the $k$th incidence matrix of $B_d$ has size $k \times (k+1),$ and is given by 
    \[ F_{k,d} = 
    \begin{bmatrix}
        1 & 0 & 0 & \ldots & 0 & 1 \\
        0 & 1 & 0 & \ldots & 0 & 0 \\
        \vdots & \vdots & \vdots & \ddots & \vdots & \vdots \\
        0 & 0 & 0 & \ldots & 1 & 0
    \end{bmatrix}
    = 
    \begin{bmatrix}
        1 & 0^*  & 1 \\
        0^* &  I_{k-1}  & 0^* \\
        
    \end{bmatrix}
    ,
    \]
    and for $k \ge d$, the incidence matrix $F_{k,d}$ has size $d \times d$, and is given by 
    \[ F_{k,d} = 
    \begin{bmatrix}
        1 & 0 & 0 & \ldots & 0 & 1 \\
        1 & 0 & 0 & \ldots & 0 & 0 \\
        \vdots & \vdots & \vdots & \ddots & \vdots & \vdots \\
        0 & 0 & 0 & \ldots & 1 & 0
    \end{bmatrix}
    =
    \begin{bmatrix}
        1&   0^*  & 1 \\
        &I_{d-1}   & 0^* \\
        
    \end{bmatrix}
    ,
    \]
    where $I_n$ represents the identity matrix of size $n$, and $0^*$ denotes a column or row consisting entirely of zeros.
\end{lemma}

\begin{lemma}
    \label{lem:first_level_telescop}
    For every $d \ge 2$, we have \[ F_{1,d} F_{2,d} \ldots F_{d-1,d} = 
    \begin{bmatrix}
        1 & 1  & \ldots & 1
    \end{bmatrix} 
    \in M_{1 \times d}.
    \]
\end{lemma}

\begin{proof}
    The proof is obtained by induction on $d$. For $d =2$, the result is true by the definition of the diagram $B_d$. Suppose that the 
    result holds for some $d' \ge 2$. By the definition of $B_d$, for every $1 \le k \le d'-1$, \[ F_{k,d'} = F_{k,d'+1}. \] Then,
    \begin{equation*}
        F_{1,d'+1}F_{2,d'+1} \ldots F_{d'-1,d'+1} F_{d',d'+1} = F_{1,d'}F_{2,d'} \ldots F_{d'-1,d'} F_{d',d'+1}.
    \end{equation*}
    By the induction hypothesis, 
    \[ F_{1,d'+1}F_{2,d'+1} \ldots F_{d'-1,d'+1} F_{d',d'+1} = 
    \begin{bmatrix}
        1 & 1  & \ldots & 1
    \end{bmatrix}
    F_{d',d'+1} =
    \begin{bmatrix}
        1 & 1  & \ldots & 1
    \end{bmatrix}
    \in M_{1,d'+1}.
    \]
    This completes the proof.
\end{proof}

Now, we study the product of the matrices $F_{k,d}$ for $k \geq d$. Put 
\[ F_d \=  
\begin{bmatrix}
    1 & 0^* & 1 \\
    & I_{d-1} & 0^*
\end{bmatrix}
.\] 
Then $\prod_{i=0}^{n-1} F_{d+i,d} = (F_d)^n$. In the following lemma, we write the $n$th power of the matrix $F_d$ in 
terms of the cutting time sequence associated with the kneading map $Q_d(k) = \max \{0, k-d\}$. Throughout the rest of this section, we work with a fixed value of $d$ and omit explicitly indicating dependence on it when it is clear from context.

Let us start by extending the sequence of cutting times to some negative values. We do this so that the recurrence $S_d(k) = S_d(k-1) + S_d(k-d)$ holds uniformly for all indices appearing in the matrix formula below. Define

\begin{enumerate}
    \item[(i)] $S_d(-2d+2)=0$,
    \item[(ii)] $S_d(-2d+1)=1$,
    \item[(iii)] $S_d(-3d+3)=S_d(-3d+4)=S_d(-3d+5)=\cdots=S_d(-2d)=0$.
\end{enumerate}

With this extension, the linear recurrence equation \cref{eq:cut_times_formula} can be written as 
\begin{equation}
    \label{eq:exte_recur}
    S_d(k) = S_d(k-1) + S_d(k-d),
\end{equation}
for every $k \ge -2d + 3$.

\begin{lemma}
    \label{lem:powers_of_F_d}
    For every $j \ge 1$, we have
     \[  F_d^j = 
     \begin{bmatrix}
        S_d(j-d+1)      & S_d(j-2d+2)   & S_d(j-2d+3)   & \cdots    & S_d(j-d)      \\
        S_d(j-d)        & S_d(j-2d+1)   & S_d(j-2d+2)   & \cdots    & S_d(j-d-1)    \\
        S_d(j-d-1)      & S_d(j-2d)     & S_d(j-2d+1)   & \cdots    & S_d(j-d-2)    \\
        \vdots          & \vdots        & \vdots        & \ddots    & \vdots &      \\
        S_d(j-2d+2)     & S_d(j-3d+3)   & S_d(j-3d+4)   & \cdots    & S_d(j-2d+1)
    \end{bmatrix}.\]
\end{lemma}
\begin{proof}
    We prove the lemma by induction. For $j=1$, this is true by the definition of the extended cutting times. Suppose that
    the claim is true for $j=k$. Note that 
    \begin{align*}
        F_d^{k+1}& =F_d^k \cdot F_d \\ 
        &= \begin{bmatrix}
            S_d(k-d+1)+S_d(k-2d+2)  & \cdots & S_d(k-d)+S_d(k-2d+1) \\
            S_d(k-d+1)  & \cdots & S_d(k-d) \\
            \vdots   & \ddots & \vdots \\
            S_d(k-2d+3)  & \cdots & S_d(k-2d+2)
        \end{bmatrix} .
    \end{align*}
    The lemma follows from \cref{eq:exte_recur}.
\end{proof}

Knowing the incidence matrices of the Bratteli diagram $B_d$, we can obtain explicitly the number of finite paths in $B_d$
from $v_0$ arriving at any other vertex in the diagram. For $ 1 \le \ell \le d$, put $N_j(\ell) \= N_j(v_j(\ell))$.

\begin{lemma}
    \label{lem:num_path_in_X_d}
    For every $1 \le j < d$, we have 
    \[ \vec{N}_j =
    \begin{bmatrix}
        1 & 1 & \cdots &1
    \end{bmatrix}
    \in \cM_{1 \times (j+1)}.
    \]
    For $j \geq d$, 
    \[ \vec{N}_j = 
    \begin{bmatrix}
        S_d(Q_d(j)+1) & S_d(Q_d(j)-d+2) & S_d(Q_d(j)-d+3) & \cdots & S_d(Q_d(j)) 
    \end{bmatrix}
    \in \cM_{1 \times d}.
    \]
    Thus, $N_j(1) = S_d(j-d+1)$, and for every $2 \le \ell \le d$, $N_j(\ell) =S_d(j+\ell-2d)$.
\end{lemma}

\begin{proof}
    For $1 \le j < d$, the lemma follows from \cref{lem:first_level_telescop}. For $j \ge d$, writing $j = d + j'$ 
    we have $\vec{N}_j = \vec{N}_{d-1}F_d^{j'+1}$. Then the lemma follows from \cref{lem:inc_mtx}, and
    \cref{lem:powers_of_F_d}.
\end{proof}

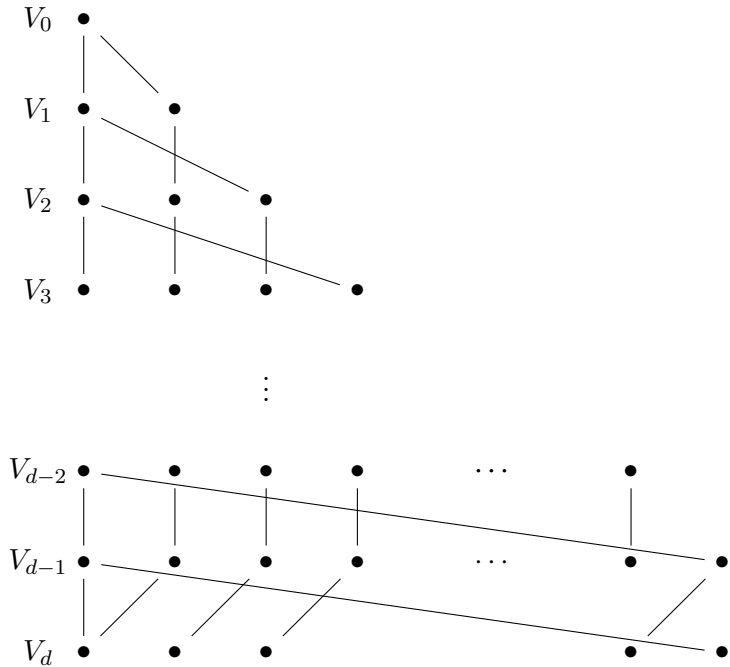
\begin{figure}[b]
    \begin{tikzpicture}[scale=1.2]
        \node at (-0.5,16) {$V_0$};
        \node at (0,16) (a1) {$\bullet$};
        \node at (-0.5,15) {$V_1$};
        \node at (0,15) (a2) {$\bullet$}; \node at (1,15) (b2) {$\bullet$};
        \draw (a1) -- (a2); \draw (a1) -- (b2); 
        \node at (-0.5,14) {$V_2$};
        \node at (0,14) (a3) {$\bullet$}; \node at (1,14) (b3) {$\bullet$}; \node at (2,14) (c3) {$\bullet$};
        \draw (a2) -- (a3); \draw (a2) -- (c3); \draw (b2) -- (b3); 
        \node at (-0.5,13) {$V_3$};
        \node at (0,13) (a4) {$\bullet$}; \node at (1,13) (b4) {$\bullet$}; \node at (2,13) (c4) {$\bullet$}; 
        \node at (3,13) (d4) {$\bullet$};
        \draw (a3) -- (a4); \draw (a3) -- (d4); \draw (b3) -- (b4); \draw (c3) -- (c4); 
        
        \node at (2,12)  {$\vdots$};
        
        \node at (-0.5,11) {$V_{d-2}$};
        \node at (0,11) (ad-1) {$\bullet$}; \node at (1,11) (bd-1) {$\bullet$}; \node at (2,11) (cd-1) {$\bullet$}; 
        \node at (3,11) (dd-1) {$\bullet$}; \node at (4.5,11) {$\cdots$}; \node at (6,11) (ed-1) {$\bullet$};
        \node at (-0.5,10) {$V_{d-1}$};
        \node at (0,10) (a) {$\bullet$}; \node at (1,10) (b) {$\bullet$}; \node at (2,10) (c) {$\bullet$}; 
        \node at (3,10) (d) {$\bullet$}; \node at (4.5,10) {$\cdots$}; \node at (6,10) (e) {$\bullet$}; \node at (7,10) (f) {$\bullet$};
        \draw (ad-1) -- (a); \draw (ad-1) -- (f); \draw (bd-1) -- (b); \draw (cd-1) -- (c); \draw (ed-1) -- (e);
        \draw (dd-1) -- (d);
        \node at (-0.5,9) {$V_{d}$};
        \node at (0,9) (A) {$\bullet$}; \node at (1,9) (B) {$\bullet$}; \node at (2,9) (C) {$\bullet$}; 
        \node at (4.5,11) {$\cdots$}; \node at (6,9) (E) {$\bullet$}; \node at (7,9) (F) {$\bullet$};
	\draw (a) -- (A); \draw (a) -- (F); \draw (b) -- (A); \draw (c) -- (B); \draw (d) -- (C); \draw (f) -- (E);
    \end{tikzpicture}
    \caption{First $d$ levels of the Bratteli diagram associated with the collection $\{M_{d,k}\}_{k \geq 0}$ of covers 
    of $T_d|_{\omega(c)}$. Each collection of vertices is arranged in increasing order. For example, from left to right, 
    the vertices in $V_3$ are $v_3(1)$, $v_3(d-2)$, $v_3(d-1)$, and $v_3(d)$.}
    \label{fig:Bratelli_diagram_first_d_levels}
\end{figure}

Now, in order to define a partial order on $E$, it is enough to define a linear order for every $k \ge 1$ on the set
\[ t^{-1}(v_{k}(1))  = \{ (v_{k-1}(1),v_k(1)), (v_{k-1}(2),v_k(1)) \}.\] 
For $k \le d-1$, $t^{-1}(v_k(1))$ is a singleton, so ordering is trivial. For any $k \ge d$, we consider the order 
\begin{equation}
    \label{eq:part_oder_in_E}
    (v_{k-1}(1),v_k(1)) < (v_{k-1}(2),v_k(1)),
\end{equation}
see \cref{fig:Bratelli_diagram_first_d_levels}. With this partial order, the set $X_{B_d}$ has one minimal path 
$x_{min} = x_1x_2 \ldots$ with 
$x_1 = (v_0,v_1(1))$ and for $k \geq 2$, $x_k = (v_{k-1}(1),v_k(1))$, and $d$ distinct maximal paths. Thus, the Vershik 
map $V_{B_d}$ is well defined on $X_{B_d}$.


\subsection{Projecting $X_{B_d}$ onto $\omega_f(c)$.}\label{sebsection:projecting_the_Bratteli_diagram_to_the_interval}

We start by defining the projection of cylinders in $X_{B_d}$ to the domain of $f$. The idea of how to define the 
projection is the following. Recall that for every integer 
$k \ge 1$, each vertex $v_k(i) \in V_k$ corresponds to the tower $\cT_k(i)$ from $M_{d,k}$, where 
$i = 1, \max\{2, d-k+1\}, \ldots, d$. By construction, if $i \ne 1,2,$ we have $\cT_{d,k}(i) = \cT_{d,k+1}(i-1)$, for 
$i = 1$ we have $\cT_{d,k+1}(1) \subseteq \cT_{d,k}(1) \cup \cT_{d,k}(2)$, and $\cT_{d,k+1}(d) \subseteq \cT_{d,k}(1)$.
So, for each finite path $e_1 \ldots e_n$ on $B_d$, we can project the cylinder $[e_1 \ldots e_n]$ to the tower 
$t(e_n) \in V_n$.
Thus, each path from $v_0$ to $v_k(i)$ needs to be projected to an element of $\cT_{d,k}(i)$. Since we want this 
projection to conjugate the action of $V_{B_d}$ on $X_{B_d}$ with the action of $f$ on $\omega_f(c)$, each minimal 
path from $v_0$ to $v_k(i)$ needs to be projected to $J_{i,k}$, the "base" of the tower $\cT_{d,k}(i)$, and each 
successor should be projected to the corresponding image of the base set.

Now we define the projection explicitly. For every integer $k \ge d$ and every $1 \le i \le d$, denote by $\cN_k(i)$
the set of all finite paths on $B_d$ from $v_0$ to $v_k(i)$. For $e_1 \ldots e_k \in \cN_k(i)$, put 
\begin{equation}
    \label{eq:eta_k}
    \eta(e_1 \ldots e_k) \= \# \{ a_1 \ldots a_k \in \cN_k(i) \colon a_1 \ldots a_k < e_1 \ldots e_k \}.
\end{equation}
Thus, $\eta (e_1 \ldots e_k)$ is the number of finite paths smaller than $e_1 \ldots e_k$ with respect with the lexicographic order 
induced by the edge ordering \cref{eq:part_oder_in_E}. Since $v_n(1)$ for $n \ge d$ is the only vertex with more than 
one incoming edge, we only need to keep track on the edges $e_m$ for which $t(e_m) = v_m(1)$. For $m \ge d$, put 
\[ \theta_m(e_1, \ldots e_k) = 
\begin{cases}
    1 & \text{ if } e_m = (v_{m-1}(2),v_m(1)), \\
    0 & \text{ otherwise }.
\end{cases} \]
Then we can write
\begin{equation}
    \label{eq:eta_k_formula}
    \eta(e_1 \ldots e_k) = \sum_{m = d}^{k} \theta_m(e_1 \ldots e_k)N_{m-1}(1) =\sum_{m = d}^{k} \theta_m(e_1 \ldots e_k)S_d(m-d), 
\end{equation}
where the last equality comes from \cref{lem:num_path_in_X_d}. This is because, at level $m$, choosing the largest incoming edge into $v_m(1)$ contributes exactly the number of paths from $v_0$ to $v_{m-1}(1)$, as the tail after level $m$ is fixed in lexicographic order.
Thus, if we let $i(e_k) \in \{ 1, 2, \ldots, d \}$ be so
that $t(e_k) = v_k(i(e_k))$, we define the \emph{projection} of the cylinder $[e_1 \ldots e_k]$ to $I$ by 
\begin{equation}
    \label{eq:cylinder_projection}
    a([e_1 \ldots e_k]) \= J_{i(e_k),k}^{\eta(e_1 \ldots e_k)}.
\end{equation}
By the construction of the Bratteli diagram $B_d$, using the tower refinement from $M_{d,k}$ to $M_{d,k+1}$, we have that
if $e_1 \ldots e_k e_{k+1}$ is a finite path on $B_d$, then $a([e_1 \ldots e_ke_{k+1}]) \subseteq a([e_1 \ldots e_k])$.
Since our map $f$ does not have wandering intervals, we can define the projection 
\[\pi \colon X_{B_d} \longrightarrow \omega_f(c)\] by
\begin{equation}
    \label{eq:Projection_definition}
    \pi(e_1e_2 \ldots) \= \bigcap_{n=1}^{\infty} a([e_1 \ldots e_n]).
\end{equation}
By item $(i)$ in \cref{theo:covers_of_the_omega_set}, this projection map is well defined. For $(e_k)_{k \ge 1} \in X_{B_d}$, we put $\theta_m((e_k))_{k\ge 1} \= \theta_m(e_1 \ldots e_m)$, and $\eta_m((e_k)_{k \ge 1}) = \eta(e_1 \ldots e_m)$.

Before proving \cref{theo:B-V_system_from_the_cover}, we will describe the structure of maximal paths.
First we observe that, since for every $\ell \in \{ 1,2, \ldots ,d\}$ there is a unique path on $B_d$ from $v_0$ to $v_{d-1}(\ell)$, each maximal path is characterized by the incoming edge at the vertex set $V_{d-1}.$ For every $\ell \in \{1,2, \ldots,d\}$, we will denote by $x^{\max,\ell} = (x_k(\ell))_{k \ge 1} \in X_{B_d}$ the unique maximal path with $t(x_{d-1}(\ell)) = v_{d-1}(\ell)$.

\begin{lemma}
    \label{lem:maximal_paths_in_B_d}
    For every $\ell \in \{1,2, \ldots, d\}$ and every $k \ge 1$ for which $\ell + kd -2 \neq d-1$, we have \[ \theta_{\ell + kd-2}(x^{\max,\ell}) = 1. \] And thus,
    \[\eta_{\ell +kd -2}(x^{max,\ell}) = S_d(\ell +(k-1)d -1)-1.\] 
\end{lemma}

\begin{proof}
    First, observe that for every $n \ge d-1$, there is a unique maximal path on $B_d$ from $v_n(1)$
    to $V_{n+d}$. We can write this path explicitly, starting from $v_k(1)$ there are two outgoing edges 
    $(v_k(1),v_{k+1}(1))$ and $(v_k(1),v_{k+1}(d))$ with only the former being a maximal path, and from
    $v_{k+1}(d)$ there is a unique path landing at the vertex $v_{k+d}(1)$. 

    Now, for each $ \ell \in \{1,2, \ldots, d\}$, the maximal path $x^{\max,\ell}$ enters the set $V_{d-1}$ at the 
    vertex $v_{d-1}(\ell)$ and thus, for $\ell \neq 1$ the smallest indexed edge in $x^{\max,\ell}$ with terminal vertex
    $v_k(1)$ for $k >d-1$ is $x_{\ell+d-2}(\ell)$. In particular, $\theta_{\ell + kd-2}(x^{\max,\ell}) = 1$. 
    For $\ell =1$, we have that the smallest indexed edge in $x^{\max,1}$ with terminal vertex $v_k(1)$ for $k >d-1$
    is $x_{2d-1}(1)$, in particular, $\theta_{kd-1}(x^{\max,1}) = 1$. This proves the first part of the lemma.

     We prove the second part using induction on $k$. By definition, for $k =1$, we obtain 
     \[  \eta_{\ell + d -2}(x^{\max,\ell}) =  S_d(\ell-2) = S_d(\ell-1) -1, \]
     so the result is true in this case. Now, suppose it is true for some $k \ge 1$. Then
     \begin{align*}
         \eta_{\ell + (k+1)d -2}(x^{\max,\ell}) &= \sum_{m=0}^{k} S_d(\ell +md -2) \\ 
         &= S_d(\ell +(k-1)d-1)-1 + S_d(\ell+kd-2) \\ 
         &=S_d(\ell +kd -1)-1,
     \end{align*}
     where in the second equality, we use the induction hypothesis, and in the third equality, we use 
     \cref{eq:cut_times_formula}. This concludes the proof of the lemma.
\end{proof}

Now we are ready to prove \cref{theo:B-V_system_from_the_cover}.

\begin{proof}[Proof of \cref{theo:B-V_system_from_the_cover}]
    We will prove that the map $\pi$ defined in \cref{eq:Projection_definition} 
    is a semi-conjugacy between the Bratteli-Vershik system $(X_{B_d}, V_{B_d})$ and $(\omega_f(c),f)$. 

    The continuity of the $\pi$ follows by its definition, the fact that the cylinders form a basis for the topology
    of $X_{B_d}$, and \cref{theo:covers_of_the_omega_set}.

    Now we prove that the map $\pi$ is onto. By \cref{eq:cylinder_projection}, we have that for every $k \in \N$
    and every $i = 1, \max\{2, d-k+1\}, \ldots, d$,
    \[ \pi ( \{ x \in X_{B_d} \colon x_1\ldots x_k \in \cN_k(i)\}) = \cT_{d,k}(i) \cap \omega_f(c). \]
    Taking the union over all $i = 1, \max\{2, d-k+1\}, \ldots, d$ above, we obtain
    \[ \pi(X_{B_d}) = M_{d,k} \cap \omega_f(c) = \omega_f(c), \] and thus the map $\pi$ is onto.
    
    So, we are left to prove that the map $\pi$ conjugates the actions of $V_{B_d}$ and $f$. Thus, we want to prove that
    for any $(e_j)_{j \ge 1} \in X_{B_d}$, we have 
    \begin{equation}
        \label{eq:conjugation_f_and_B_Bd}
        \pi(V_{B_d}(e_j)_{j \ge 1}) = f(\pi(e_j)_{j \ge 1}).
    \end{equation}
    
    We will consider two cases. 

    \noindent \textbf{Case 1.} Consider $x^{\max,\ell} = (x_k(\ell))_{k \ge 1}$ with $\ell \in \{1,2, \ldots, d\}$. Recall that every maximal path is mapped by $V_{B_d}$ to the unique minimal path on $B_d$, thus $V_{B_d}(x^{\max,\ell}) = x^{\min}$. Write $x^{\min} = (x_k^0)_{k \ge 1}$. First, we observe that for every $k \ge 1$
    \[ \pi([x_1^0 \ldots x_k^0]) = I_k \cap \omega_f(c), \]
    and since $\cap_{k \ge 1} I_k = \{ c\}$, we have that $ \pi(x^{\min}) = c$. Now, by \cref{lem:maximal_paths_in_B_d}, for every $\ell \in \{1 ,\ldots, d\}$ and every $k \ge 1$ for which $\ell + kd -2 \ne d-1$, we have
    \[ \pi([x_1(\ell) \ldots x_k(\ell)]) = I_{\ell +kd -2}^{S_d(\ell +(k-1)d-1)-1} \ \cap \ \omega_f(c), \]
    and thus
    \[f(\pi(x^{\max,\ell})) \subset \bigcap_{k \ge 1} [c_{S_d(\ell+ (k-1)d-1)}, c_{S_d(\ell +kd-1)}] \ \cap \ \omega_f(c) = \{c\}.\] This proves \cref{eq:conjugation_f_and_B_Bd} in this case.

    \noindent \textbf{Case 2.} Consider $(e_k)_{k \ge 1} \in X_{B_d} \setminus X^{\max}$. Let $j_0 \ge d$ be the 
    smallest integer for which $e_{j_0}$ is not a maximal edge, thus $e_{j_0} = (v_{j_0-1}(1),v_{j_0}(1))$ and 
    $e_1 \ldots e_{j_0-1}$ is a maximal path from $v_0$ to $v_{j_0-1}(1)$. By \cref{lem:maximal_paths_in_B_d}, with 
    $\ell + kd-2$ replaced by $j_0 -1$, we get
    \[\eta_{j_0}((e_j)_{j \ge 1}) = S_d(j_0-d)-1.\]
    Writing $x \= \pi(e_j)_{j \ge 1}$, the above implies that $x \in I_{j_0}^{S_d(j_0-d)-1}$, and
    \[f(x) \in I_{j_0}^{S_d(j_0-d)} = f(\pi([e_1 \ldots e_{j_0}])) \in \cT_{d,j_0}(1),\]
    and thus
    \begin{equation}
        \label{eq:conj_eq_1}
        f(x) \in f(\pi([e_1 \ldots e_{j_0 + k}]))
    \end{equation}
    for every $k \ge 0$.
     
    Put $V_{B_d}(e_j)_{j \ge 1} = g_1 \ldots g_{j_0}e_{j_0+1} \ldots$. Thus, $g_{j_0} = (v_{j_0-1}(2), v_{j_0}(1))$ and 
    $g_1 \ldots g_{j_0-1}$ is the unique minimal path from $v_0$ to $v_{j_0-1}(2)$. Then, 
    $\theta_n(g_1 \ldots g_{j_0}) = 0$ for $n = 1, \ldots j_0-1$, and $\theta_{j_0}(g_1 \ldots g_{j_0}) = 1$, so 
    \[ \eta_{j_0}(V_{B_d}(e_j)_{j \ge 1}) = S_d(j_0-d) = \eta_{j_0}((e_j)_{j \ge}) + 1. \]
    This implies that for every $k \ge d$,
    \begin{equation}
        \label{eq:conj_eq_2}
        \pi (V_{B_d}((e_j)_{j \ge 1})) \in f(\pi([e_1 \ldots e_k])).
    \end{equation}
    Since $|f(\pi([e_1 \ldots e_k]))| \rightarrow 0$ as $k \to \infty$, \cref{eq:conj_eq_1} and \cref{eq:conj_eq_2} imply 
    \[f(\pi (e_j)_{j \ge 1}) = \pi(V_{B_d}((e_j)_{j \ge 1})).\]
    This concludes the proof of the theorem.
\end{proof}

\subsection{The push-forward measure.}
In this section, we will prove \cref{theo:measure_of_the_covers}. We start by computing the unique ergodic $V_{B_d}-$invariant probability measure of the Bratteli-Vershik system $(X_{B_d},V_{B_d})$ described in \cref{subsection:BV_system_associated_to_Mdk}. The ideas we use for this are standard and we include a detailed computation for the convenience of the reader; we refer the reader to \cite{Bezuglyi_Kwiat_med_Solomyak_Inv_meas_stationary_Brat_diagrams_ETDS_2010} or \cite[Chapter 6]{Bruin2022_Top_and_erg_symb_dyn}, and references therein. We push this measure forward to our interval map using the projection defined in \cref{sebsection:projecting_the_Bratteli_diagram_to_the_interval}.

For every $\ell \in \{1, \ldots, d\}$, write $x^{\max,\ell} = (x_j(\ell))_{j \ge 1}$. For $j \ge d-1$, we have 
\begin{equation}
    \label{eq:paths_at_a_vertex}
    \{ e_1 \ldots e_j \colon t(e_j) = t(x_j(\ell))\} = 
    \{ V_{B_d}^{-n}(x^{\max, \ell}) \colon n \in \{ 0, \ldots , N_j(\ell)-1 \} \}
\end{equation}
and thus
\begin{align}
    \begin{aligned}
    \label{eq:KR_partition_of_X_B}
    \cP_j &\= \{ V_B^{-n}([x_1(\ell) \ldots x_j(\ell)]) \; \colon n \in \{0, \ldots , N_j(\ell) -1\}, \ell \in \{1, \ldots, d\} \} \\
    &= \{ [e_1\ldots e_{j}] \; \colon e_1\ldots e_{j} \text{ is a path starting at } v_0 \}, 
    \end{aligned}
\end{align}
is a partition  of $X_{B_d}$ into clopen sets, and $\cP_{j+1}$ is finer than $\cP_j$. It follows that each invariant measure 
$\mu$ of the system $(X_{B_d},V_{B_d})$ is determined by its values on the elements of the partition $\cP_j$.

Denote by $\nu_d$ the unique ergodic $V_{B_d}$-invariant probability measure of the system $(X_{B_d},V_{B_d})$. As 
mentioned before, the measure $\nu_d$ is uniquely determined by its values on the cylinder sets
$[e_1 \ldots e_j]$, where $j \in \N$, and $e_1 \ldots e_j$ is a finite path on $X_{B_d}$ starting at $v_0$. Moreover, since the measure $\nu_d$ is $V_{B_d}-$invariant, we have that for any fixed level $j \ge d-1$ (when the diagram becomes stationary), all cylinders entering the same vertex $v_j(\ell)$ have the same measure (see \cite[Section 2]{Bezuglyi_Kwiat_med_Solomyak_Inv_meas_stationary_Brat_diagrams_ETDS_2010}).
To compute the value of $\nu_d$ on cylinders, we use \cite[Remark 3.9]{Bezuglyi_Kwiat_med_Solomyak_Inv_meas_stationary_Brat_diagrams_ETDS_2010}. For $j \ge d-1$, 
put $\nu_d(j,\ell) \= \nu_d([x_1(\ell) \ldots x_j(\ell)])$ and $\vec{\nu}_d(j) \= (\nu_d(j,1), \nu_d(j,2), \ldots, \nu_d(j, d)) \in \R^d$. Then 
\[  \vec{\nu}_d(j) = C \beta_d^{-j} \vec{w}(\beta_d), \]
where $\beta_d$ is the Perron-Frobenius eigenvalue of the incidence matrix $F_{k,d}$ of the stationary levels of Bratteli diagram $X_{B_d}$, $\vec{w}(\beta_d)$ is the Perron-Frobenius eigenvector associated with the eigenvalue $\beta_d$, and $C$ is a normalization constant. From \cref{lem:powers_of_F_d}, we have that $\beta_d$ is the largest real root of the polynomial $x^d - x^{d-1}-1$, and 
\[\vec{w}(\beta_d) = (1, \beta_d^{-1}, \dots, \beta_d^{-(d-1)}).\] 
Now, in order to compute the normalization constant $C$, we can use the fact that there is a unique path on $B_d$ from $v_0$ to $v_d(\ell)$ for every $ \ell \in \{1, \ldots, d\}$. More explicitly,
\[ 1 = \sum_{i=1}^d \nu_d([x_1(i) \ldots x_{d-1}(i)]) = C \beta_d^{-(d-1)} \sum_{i=1}^d\beta_d^{-(i-1)}.\]
We observe that 
\[
\sum_{i=1}^d \beta_d^{-(i-1)}=\sum_{i=0}^{d-1}\beta_d^{-i}
=\frac{1-\beta_d^{-d}}{1-\beta_d^{-1}}
=\frac{\beta_d^{d}-1}{\beta_d^{d-1}(\beta_d-1)}.
\]
Using $\beta_d^{d}=\beta_d^{d-1}+1$, we have $\beta_d^{d}-1=\beta_d^{d-1}$, hence
\[
\sum_{i=1}^d \beta_d^{-(i-1)}=\frac{1}{\beta_d-1}.
\]
Therefore,
\[
1=\sum_{i=1}^d \nu_d([x_1(i)\ldots x_{d-1}(i)])
= C\beta_d^{-(d-1)}\sum_{i=1}^d \beta_d^{-(i-1)}
= C\frac{\beta_d^{-(d-1)}}{\beta_d-1}.
\]
Finally, since $\beta_d^{d}=\beta_d^{d-1}+1$ implies $\beta_d^{d-1}(\beta_d-1)=1$, we get
$\frac{\beta_d^{-(d-1)}}{\beta_d-1}=1$, and hence $C=1$, and we can write, for every $j \ge d-1$ and every $1 \leq  \ell \le d$, 
\begin{equation}
    \label{lem:meas_charact}
    \nu_d(j,\ell) = \frac{1}{\beta_d^{j + \ell -1}}.
\end{equation}
Finally, using \cref{lem:num_path_in_X_d}, for every $j \ge d-1$ and every path $e_1 \ldots e_j$ on $B_d$ starting from $v_0$, we have 
\begin{equation}
    \label{eq:acip_on_X_B_d}
    \nu_d([e_1 \ldots e_j]) = 
    \begin{cases}
        \beta_d^{-j} & \text{ if } \ t(e_j) =v_j(1) \\
        \beta_d^{-(j+\ell -1)} & \text{ if } \ t(e_j) = v_j(\ell) \text{ and } \ell \neq 1.
    \end{cases}
\end{equation}

\begin{proof}[Proof of \cref{theo:measure_of_the_covers}]
    Let $\mu_d$ be the unique ergodic $f-$invariant probability measure of $f$ supported on $\omega_f(c)$. Since the projection $\pi \colon X_{B_d} \longrightarrow \omega_f(c)$ is a homeomorphism outside a countable set (the only points with more than one coding are those whose path is eventually maximal, corresponding to backward orbits of $c$; this set is countable, this set of eventually maximal paths are projected onto the set of preimages of $c$) and thus is a measurable homeomorphism, we have that the push-forward measure $\pi_*(\nu_d) = \nu_d \circ \pi^{-1}$ is $f$-invariant and ergodic; by unique ergodicity it equals $\mu_d$. 
    Thus, the theorem follows from \cref{eq:acip_on_X_B_d}.
\end{proof}

\section{The Hausdorff dimension of $\omega_{T_{a_d}}(0)$} \label{sec:HD_proof}

Recall that $a_d$ is the unique parameter in $(0,2]$ for which the map $T_{a_d}=a_d(1-|x|)-1$ has kneading map 
\[Q_d(k) = \max \{0, k-d\}.\]
Here $c = 0$ is the turning point of $T_{a_d}$, and $c_n = T_{a_d}^n(c).$
In this section, we will prove \cref{theo:HD}.

Let $M_{d,k}$ be the cover defined in \cref{eq: union} for the map $T_{a_d}(x)$. For every $k \ge 1 $, set $D_k \= [c_{S_d(k)},c]$, and
put \[ \delta_k \= \max \{ |J_{i,k}^n| \colon i = 1, \ldots, d, \; \text{ and } \; n = 0, \ldots h_{i,k}-1 \}. \]
Clearly, $\delta_k \to 0$ as $k \to \infty$. 

Note that by definition, $D_{k'} \subset J_{1,k}$ for every $k' \ge k \ge 1$. Moreover, by \cref{eq:def_J_ik}, for 
$j = d$, we obtain that for every $0 < k < k'$,
\begin{equation}
    \label{lem: j and d and i}
    |J_{d,k'}| \le |D_{k'-1}| \le |J_{1,k}|.
\end{equation}
The following lemma gives us an estimate of the recurrence rate of the turning point for $T_{a_d}$.
\begin{lemma}\label{lem: diameter}
The limit
\[\lim_{k \to \infty}a_d^{S_d(k+d)} \prod_{j=1}^{d-1} |D_{k+j}|\]
exists and is strictly positive.
\end{lemma}

\begin{proof}
Let $k > d$.
By definition, we have that $T_{a_d}(D_k)=J_{1,k}^1$, by \cref{lem: injectivity}, we have 
\[T_{a_d}^{S_d(k+1-d)}(D_k)=\left[c_{S_d(k+1)},c_{S_d(k+1-d)} \right].\]
By \cref{eq:cut_time}, $c_{S_d(k+1)}$ and $c_{S_d(k+1-d)}$ are on opposite sides of $c$.
Then $D_{k+1} \cap D_{k+1-d}=\{c\}$, so \[T_{a_d}^{S_d(k+1-d)}(D_k)=D_{k+1} \cup D_{k+1-d}.\]
Since $|T_{a_d}^{S_d(k+1-d)}(D_k)|=a_d^{S_d(k+1-d)}|D_k|=|D_{k+1}|+|D_{k+1-d}|$, we have
\[a_d^{S_d(k+1-d)}=\frac{|D_{k+1}|}{|D_k|}+\frac{|D_{k+1-d}|}{|D_k|}.\] Set $\nu_k:=|D_k|/|D_{k+1}|$. 
Then, 
\begin{equation}
    \label{eq:diam_estimate_1}
    a_d^{S_d(k+1-d)}=\frac{1}{\nu_{k}}+\prod_{i=1}^{d-1}\nu_{k+i-d}.
\end{equation}
By \cref{eq:close_retur_symm_tent_map}, $\nu_k > 1$, so $0 < \nu_k^{-1} < 1$. Since $a_d > 1$, we have
$a_d^{S_d(k+1-d)} \to \infty$ as $k \to \infty$.
Then, $1 - a_d^{-S_d(k+1-d)} \prod_{i=1}^{d-1}\nu_{k+i-d} \to 0$ as $k \to \infty$.
Put 
\begin{equation}
    \label{eq:diam_estimate_2}
    C_{k+1}:=a_d^{-S_d(k+1-d)}\prod_{j=1}^{d-1}\nu_{k+j-d}.
\end{equation}
Then, $0 < C_k < 1$ and $C_k \rightarrow 1$.
By the definition of $\nu_k$, we have
\[
\prod_{i=2}^k a_d^{S_d(i)}C_{d+i} =\prod_{i=2}^k \prod_{j=1}^{d-1}\nu_{i+j-1}
=\prod_{j=1}^{d-1}\left[\nu_{j+1}\nu_{j+2}\cdots \nu_{k+j-1} \right] \\
=\prod_{j=1}^{d-1}\frac{|D_{j+1}|}{|D_{k+j}|}.
\] 
By \cref{lem:sum_cut_times}, 
\[a_d^{S_d(k+d)-S_d(d-1)-S_d(1)-S_d(0)} \prod_{j=1}^{d-1} |D_{k+j}| = \prod_{i=2}^k C_i^{-1} \prod_{j=1}^{d-1}|D_j|.\]
Thus, \[a_d^{S_d(k+d)} \prod_{j=1}^{d-1}|D_{k+j}| = a_d^{S_d(d-1)+S_d(1)+S_d(0)} \left[ \prod_{j=2}^k C_j \right]^{-1} \prod_{j=1}^{d-1}|D_{j+1}|.\]
From the above, in order to prove the lemma, it is enough to prove that $\prod_{j=2}^k C_j$ converges to a strictly positive number as $k \to \infty$. For this purpose, we prove that $\sum_{j \ge 2}|1-C_j| < \infty$.

Dividing \cref{eq:diam_estimate_1} by $a_d^{S_d(k+1-d)}$, and using \cref{eq:diam_estimate_2} give
\[1=\frac{1}{a_d^{S_d(k+1-d)}\,\nu_k}+C_{k+1},\]
hence
\[1-C_{k+1}=\frac{1}{a_d^{S_d(k+1-d)}\,\nu_k}\le a_d^{-S_d(k+1-d)}.\]
Therefore, it suffices to show that $\sum_{n\ge 0} a_d^{-S_d(n)}<\infty$.

Using the recurrence $S_d(n+1)=S_d(n)+S_d(n+1-d)$ (for $n\ge d-1$) and the fact that $S_d(m)\ge 1$ for all $m\ge 0$,
we obtain, for every $n\ge d-1$,
\[S_d(n+1)=S_d(n)+S_d(n+1-d)\ge S_d(n)+1.\]
Iterating yields $S_d(n)\ge n+O(1)$, and thus, since $a_d>1$,
\[\sum_{n\ge 0} a_d^{-S_d(n)}<\infty.\]
Consequently,
\[\sum_{j\ge 2} |1-C_j|=\sum_{j\ge 2} (1-C_j)<\infty.\]
This completes the proof.
\end{proof}
The following lemma gives us an estimate of $\delta_k$.
\begin{lemma}
    \label{lemm:delta_k_estimate}
    For every $k \in \N$ large enough, we have 
    \[ \delta_k = a_d^{S_d(k+1-d)-1}|D_k|. \]
\end{lemma}

\begin{proof}
    Fix $k \in \N$ large enough.
    By \cref{eq:def_J_ik}, for each $2 \le j \le d$, and each $0 \le i < S_d(k+j-2d)$,
    \begin{equation}
        \label{delta_k_esti_1}
        |J_{j,k}^i| = |J_{1,k+j-1}^{S_d(k+j-d-1)+i}| = a_d^{S_d(k+j-d-1)+i}|D_{k+j-1}|.
    \end{equation}
    Also, for every $1 \le i < S_d(k+1-d)$,
    \begin{equation}
        \label{delta_d_esti_2}
        |J_{1,k}^i| = a_d^i|D_k|.
    \end{equation}
    Since $a_d > 1$, we have 
    \[ \delta_k = \max \{|J_{1,k}|, a_d^{S_d(k+1-d)-1}|D_k|, a_d^{S_d(k+2-d)-1}|D_{k+1}|, \ldots, 
    a_d^{S_d(k)-1}|D_{k+d-1}|\}. \]
    Now, since for every $n \in \N$
    \[ |T_{a_d}^{S_d(k+n-d)}(D_{k+n-1})| = a_d^{S_d(k+n-d)}|D_{k+n-1}| = |D_{k+n}| + |D_{k+n-d}|, \]
    we have 
    \begin{multline*}
       a_d^{S_d(k+n-d)-1}|D_{k+n-1}|- a_d^{S_d(k+n+1-d)-1}|D_{k+n}| = a_d^{-1} [ 
        (|D_{k+n}| - |D_{k+n+1}|) +  \\ 
        (|D_{k+n-d}| - |D_{k+n+1-d}|) ].
    \end{multline*}
    By \cref{eq:close_retur_symm_tent_map}, we conclude that 
    \[ a_d^{S_d(k+n-d)-1}|D_{k+n-1}|- a_d^{S_d(k+n+1-d)-1}|D_{k+n}| > 0, \]
    and thus 
    \[ a_d^{S_d(k+n-d)-1}|D_{k+n-1}| > a_d^{S_d(k+n+1-d)-1}|D_{k+n}|. \]
    By the above, we have 
    \[ \delta_k = \max \{|J_{1,k}|, a_d^{S_d(k+1-d)-1}|D_k|\}.  \]
    Finally, by \cref{eq:I_k_interval}, and since $a_d > 1$, then for every $k \in \N$ large enough, we have 
    \[ |J_{1,k}| \leq 2|D_k| \leq a_d^{S_d(k+1-d)-1}|D_k|. \] This concludes the proof of the lemma.
\end{proof}

\begin{proof}[Proof of \cref{theo:HD}] 
    Since each $M_{d,k}$ is the union of $S_d(k)$ intervals of length at most $\delta_k$, we have that, for each $\alpha > 0$, 
    \[ \cH_{\delta_k}^{\alpha}(\omega_{T_{a_d}}(0)) \leq S_d(k) \delta_k ^{\alpha}. \]
    
    Now,
    
    \begin{align*}
        a_d^{S_d(k+1-d)-1}|D_k| &= a_d^{-1} (|D_{k+1}| + |D_{k+1-d}|) \\
        &< 2a_d^{-1}|D_{k+1-d}| \\
        &< 2 a_d^{-1} a_d^{-S_d(k+2-2d)}.
    \end{align*}
    And thus
    \begin{align*}
        \cH_{\delta_k}^{\alpha} (\omega_{T_{a_d}}(0)) &\leq S_d(k) 2^{\alpha} a_d^{-\alpha}
        a_d^{-\alpha S_d(k+2-2d)} \\
        &= \left( \frac{2}{a_d} \right)^{\alpha} S_d(k) a_d^{-\alpha S_d(k+2-2d)} \\
        &\le C \left( \frac{2}{a_d} \right)^{\alpha} \beta_d^{k+2d} a_d^{-\alpha S_d(k+2-2d)} \\
        &= C\left( \frac{2}{a_d} \right)^{\alpha} \left( \frac{\beta_d^{k+2d / S_d(k+2-2d)}}{a_d^{\alpha}} 
        \right)^{S_d(k+2-2d)},
    \end{align*}
    where $C>0$ is a constant, and $\beta_d$ is the largest real solution of the equation $x^d -x^{d-1}-1 = 0$.
    Since \[ \frac{k+2d}{S_d(k+2-2d)} \to 0 \] as $k \to \infty$, we conclude that 
    \[\cH_{\delta_k}^{\alpha} (\omega_{T_{a_d}}(0)) \longrightarrow 0\] as $k \to \infty$, for any $\alpha > 0$. Thus, 
    \[ \HD( \omega_{T_{a_d}}(0)) = 0. \]
\end{proof}

\bibliography{Biblio}{}
\bibliographystyle{alpha}
\end{document}